\documentclass[12pt]{article}
\usepackage[a4paper,mag=1000,includefoot,left=2cm,right=2cm,top=2cm,bottom=2cm,headsep=1cm,footskip=1cm]{geometry}
\usepackage[cp1251]{inputenc}
\usepackage[T2A]{fontenc}
\usepackage{amsfonts,amsmath,amssymb,amsthm}
\usepackage{bm}
\usepackage{cite}
\usepackage{pdflscape}
\usepackage{pgfplotstable,booktabs}
\usepackage{longtable,supertabular,multirow,array}
\usepackage{subfigure}
\def\ve{\varepsilon}

\def\urho{\underline{\rho}}
\def\orho{\overline{\rho}}
\def\leq{\leqslant}
\def\geq{\geqslant}
\def\half{{\textstyle\frac12}}
\def\onetwelve{{\textstyle\frac{1}{12}}}
\def\*#1{\mathbf{#1}}
\newtheorem{theorem}{\indent Theorem}
\newtheorem{lemma}{\indent Lemma}
\newtheorem{remark}{\indent Remark}

\begin{document}
\begin{center}
{\large\textbf{On Properties of Compact 4th order Finite-Difference\\[2mm] Schemes for the Variable Coefficient Wave Equation}}
\end{center}
\medskip\par\noindent{{Alexander Zlotnik$^{a,b}$
\footnote{\small Corresponding author.\\
 E-mail addresses: \text{azlotnik@hse.ru} (A. Zlotnik), \text{raimondas.ciegis@vgtu.lt} (R. \v{C}iegis)},
Raimondas \v{C}iegis$^{c}$}
\smallskip\par\noindent$^{a}$
{\small Higher School of Economics University, Pokrovskii bd. 11, 109028 Moscow, Russia
\par\noindent
$^b$ {\small Keldysh Institute of Applied Mathematics, Miusskaya sqr., 4, 125047 Moscow, Russia}
\par\noindent
$^c$ {\small Vilnius Gediminas Technical University, Saul{\. e}tekio al. 11, LT-10223 Vilnius, Lithuania}}}
\begin{abstract}
\noindent We consider an initial-boundary value problem for the $n$-dimensional wave equation with the variable sound speed, $n\geq 1$.
We construct three-level implicit in time and compact in space (three-point in each space direction) 4th order finite-difference schemes on the uniform rectangular meshes including their one-parameter (for $n=2$) and three-parameter (for $n=3$) families.
We also show that some already known methods can be converted into such schemes.
In a unified manner, we prove the conditional stability of schemes in the strong and weak energy norms together with the 4th order error estimate under natural conditions on the time step.
We also transform an unconditionally stable 4th order two-level scheme suggested for $n=2$ to the three-level form, extend it for any $n\geq 1$ and prove its stability.
We also give {an example} of
a compact scheme for non-uniform in space and time rectangular meshes.
We suggest simple fast iterative methods based on FFT to implement the schemes.
A new effective initial guess to start iterations is given too.
We also present promising results of numerical experiments.
\end{abstract}


\section{\large Introduction}
\label{s:1}

\par Vast literature is devoted to compact higher-order finite-difference schemes for PDEs including elliptic, parabolic, 2nd order hyperbolic and the time-dependent Schr\"{o}dinger equation, etc.
This is due to the fact that the formulas and implementation of compact schemes are not
{substantially more complex in comparison with}
the most standard 2nd order schemes but the error of compact schemes is usually several orders of magnitude less
on the same mesh leading to significantly less computational work to ensure given accuracy.

\par In recent years, the case of initial-boundary value problems for the multidimensional wave equation with the variable sound speed $c(x)$ has attracted
{a lot of} attention, see, in particular, \cite{BTT18,HLZ19,LLL19,STT19}, where much more relevant references can be found.
Among them, in papers \cite{BTT18} for 2D case and \cite{STT19} for 3D case, some three-term recurrent in time compact higher-order methods on the square spatial mesh have been constructed.
In the case $c(x)\equiv \textrm{const}$, the spectral {stability}
analysis of the methods has been given.
The methods are conditionally stable but implicit in time.
Therefore, to implement the methods, a direct method (for $n=2$) and iterative methods of the conjugate gradient and multigrid types (for $n=2,3$) have been considered and verified.
\par In the case $n=2$, another two-level vector in time 4th order method has been constructed in \cite{HLZ19}.
Here ``the vector method'' means that approximations for the solution $u$ and its weighted time derivative $\frac{1}{c^2(x)}\partial_tu$
are constructed jointly.
This method is unconditionally stable, but it exploits rather cumbersome approximations to the elliptic part of the wave equation involving triple application of a mesh Laplace operator and the inverse operators to the Numerov averages in each spatial coordinate.
Consequently, in our opinion, it can hardly be called compact.
Note that other two-level vector methods were studied, in particular, in \cite{BDS79,Z94}.

\par In the recent paper \cite{ZK20}, implicit three-level in time
{and}
compact in space ({the three-point} in each space direction) finite-difference schemes on uniform rectangular meshes have been constructed by other techniques for the initial-boundary value problem (IBVP) with the nonhomogeneous Dirichlet boundary condition for the $n$-dimensional wave equation with constant coefficients, $n\geq 1$.
The conditional stability together with 4th order error estimates have been rigorously proved for the schemes.
{Extension} of the schemes to the case of non-uniform in space and time rectangular meshes has been also given.

\par In this paper, we accomplish a generalization of compact schemes from \cite{ZK20} to the case $c(x)\not\equiv\textrm{const}$ based on a new technique related to averaging the wave equation.
Moreover, we present one-parameter (for $n=2$) and three-parameter (for $n=3$) families of compact schemes.
We also show how the methods from \cite{BTT18,STT19} can be rewritten as three-level compact schemes for the wave equation,
and they are included into these families of
compact approximations of the wave equation in the case of square meshes up to
our simpler approximation of the free term in the equation.
But notice that we use another (also implicit) approximation of the second initial condition $\partial_tu|_{t=0}=u_1$ similar and closely connected to the approximation of the wave equation itself (going back, in particular, to \cite{Z94}).
We also apply an operator technique that greatly simplifies and shortens
{derivation, presentation,}
generalization and analysis of the schemes.

\par We first consider three-level in time finite-difference schemes with
a weight $\sigma$ and the variable coefficient {$c(x)$}
in an abstract form and prove a theorem on stability of these schemes in
the strong (standard) and weak energy norms with respect to the initial data
and free {term}.
The stability is unconditional for $\sigma\geq\frac14$ and conditional for $\sigma<\frac14$.
In the latter case, practical stability conditions on the time step of the mesh are often derived by applying the spectral method in the case of the $c(x)\equiv \textrm{const}$ and then taking the maximal value of $c(x)$ as this constant.
The {presented} theorem justifies that such an approach is correct, in particular, for constructed compact schemes where $\sigma=\frac{1}{12}$.
As a corollary of the main theorem,
we rigorously prove the 4th order error estimate in the strong energy norm for constructed compact schemes.
Notice that the spectral analysis for $c(x)\not\equiv \textrm{const}$ is impossible,
and {our analysis is}
based on the energy method; moreover, namely stability theorems of the mentioned
type {allow us}
to prove rigorous error estimates.
We emphasize that the rigorous results on the 4th order approximation errors (in the standard sense), stability, error bounds and discrete energy conservation law are new
{results} in the case $c(x)\not\equiv \textrm{const}$,
and {they ensure}
{a strong} theoretical {basis}
{for such compact} schemes.

\par Next we consider the method from \cite{HLZ19} mentioned above.
Excluding the auxiliary unknown function approximating $\frac{1}{c^2(x)}\partial_tu$, we reduce it to the three-level method with the weight $\sigma=\frac14$.
We also generalize it to any $n\geq 1$ and prove its unconditional stability based on the above general stability theorem, now for $\sigma=\frac14$.

\par We also present an example of extending a three-level compact scheme for any $n\geq 1$ to the case of non-uniform in space and time rectangular meshes.
Note that compact schemes on non-uniform meshes for other equations were considered, in particular, in \cite{ChS18,JIS84,RCM14,Z15}.

\par In the case of the uniform rectangular mesh, we construct simple efficient one-step and $N$-step iterative methods to implement the schemes at each time level, with a preconditioner using FFT.
Under the stability condition, they are fast convergent, and the convergence rate is independent both on the meshes and $c(x)$, in particular, on the spread of its values, that is non-trivial and important {property}
for some applications.
We also suggest how to select {an effective initial} guess, which is close to the sought solution at each time level.
This choice is based on a simplified scheme for $\sigma=0$ and also assumes usage of FFT.
The one-step iterative method is applied for several numerical experiments.

\par The paper is organized as follows.
In Section \ref{general3level}, we state the IBVP for the wave equation with the variable sound speed and consider three-level in time finite-difference schemes with the weight $\sigma$ in general form.
We adapt {one recent} theorem {to prove the}
stability in the strong and weak energy norms for such schemes;
the {obtained stability estimates are} unconditional for $\sigma\geq\frac14$ and
conditional for $\sigma<\frac14$.
The energy conservation law for these schemes is written as well.
We also transform methods from \cite{BTT18,STT19} to the form of considered schemes.
In Section \ref{numerovschemes}, we generalize schemes from \cite{ZK20} to the case of the variable sound speed.
Moreover, we present one-parameter (for $n=2$) and three-parameter (for $n=3$) families of compact schemes and justify the 4th approximation order of the schemes.
We also compare the methods from \cite{BTT18,STT19} with the constructed schemes.
For these schemes, we prove theorems on their conditional stability and 4th order error bound in the strong energy norm.
Section \ref{uncond stable scheme} presents the three-level form of the two-level method from \cite{HLZ19}, its extension to any $n\geq 1$ and a theorem on their unconditional stability bounds together with the energy conservation law.
In Section \ref{non inif mesh}, we also demonsrate how to extend one of the compact schemes suitable for any $n\geq 1$ to the case of non-uniform in space and time rectangular meshes.

\par The last Section \ref{numerexperiments} is devoted to the fast iterative one-step and $N$-step methods to implement the constructed compact schemes on the uniform mesh including theorems on their convergence.
We also present results of numerical experiments on testing the constructed schemes and the one-step iterative method in 2D case
including the wave propagation in a three-layer 2D medium initiated by the Ricker-type wavelet (with discontinuous $c(x)$ and the $\delta$-shaped free term in the wave equation).

\section{\large Symmetric three-level method for second order hyperbolic equations with a variable coefficient and its stability theorem}
\label{general3level}
\setcounter{equation}{0}
\setcounter{lemma}{0}
\setcounter{theorem}{0}
\setcounter{remark}{0}

We consider the following initial-boundary value problem (IBVP) with the Dirichlet boundary condition for the wave equation in a generalized form
\begin{gather}
 \rho(x)\partial_t^2u(x,t)-(a_1^2\partial_1^2+\ldots+a_n^2\partial_n^2)u(x,t)
=f(x,t)\ \ \text{in}\ \ Q_T=\Omega\times (0,T);
\label{hyperb2eq}
\\[1mm]
 u|_{\Gamma_T}=g(x,t);\ \ u|_{t=0}=u_0(x),\ \ \partial_tu|_{t=0}=u_1(x),\ \ x\in\Omega.
\label{hyperb2ibc}
\end{gather}
Here  $0<\urho\leq\rho(x)$ in $\bar{\Omega}$,
$a_1>0,\ldots,a_n>0$ are constants (we take them different for uniformity with \cite{ZK20} and to distinct difference operators) and
$x=(x_1,\ldots,x_n)$, $n\geq 1$.
Also $\Omega$ is a bounded domain in $\mathbb{R}^n$,
$\partial\Omega$ its boundary and $\Gamma_T=\partial\Omega\times (0,T)$ is the lateral surface of $Q_T$.
Note that $c(x)=\frac{1}{\sqrt{\rho(x)}}$ is the variable sound speed in the case $a_1=\ldots=a_n=1$.

\par In this section, we consider a general three-level method with a weight for the IBVP \eqref{hyperb2eq}-\eqref{hyperb2ibc} with $g=0$. We present theorem on its stability and give the discrete energy conservation law
which are applied below in Sections 3 and 4 for various specific conditionally and unconditionally stable 4th order schemes to ensure their stability and the discrete energy conservation laws.
Such an approach is standard in the theory of difference schemes, for example, see \cite{S77}.

\par Let $H_h$ be a Euclidean
space of functions given on a spatial mesh endowed with an inner product $(\cdot,\cdot)_h$ and the corresponding norm $\|\cdot\|_h$, where $h$ is the parameter related to this mesh.
Let $B_h$ and $A_h$ be linear operators in $H_h$ having the properties $B_h=B_h^*>0$ and $A_h=A_h^*>0$.
As applied to the wave equation \eqref{hyperb2eq}, $B_h$ is an averaging operator and $A_h$ is an approximation to its elliptic part $-(a_1^2\partial_1^2+\ldots+a_n^2\partial_n^2)$.
For any operator $C_h=C_h^*>0$ in $H_h$, one can define the norm $\|w\|_{C_h}=(C_hw,w)_h^{1/2}$ in $H_h$ generated by it.

\par We introduce the uniform mesh $\overline\omega_{h_t}=\{t_m=mh_t\}_{m=0}^M$ on a segment $[0,T]$, with the step $h_t=T/M>0$ and
$M\geq 2$.
Let $\omega_{h_t}=\{t_m\}_{m=1}^{M-1}$ be the internal part of $\overline\omega_{h_t}$.
We introduce the mesh averages and difference operators
\[
 \bar{s}_ty=\frac{\check{y}+y}{2},\,\
 s_ty=\frac{y+\hat{y}}{2},\,\
 \bar{\delta}_ty=\frac{y-\check{y}}{h_t},\,\
 \delta_ty=\frac{\hat{y}-y}{h_t},\,\
 \mathring{\delta}_ty=\frac{\hat{y}-\check{y}}{2h_t},\
 \Lambda_ty=\delta_t\bar{\delta}_ty=\frac{\hat{y}-2y+\check{y}}{h_t^2}
\]
with $y^m=y(t_m)$, $\check{y}^{m}=y^{m-1}$ and $\hat{y}^{m}=y^{m+1}$,
as well as the operator of summation with the variable upper limit
\[
 I_{h_t}^my=h_t\sum_{l=1}^m y^l\ \ \text{for}\ \ 1\leq m\leq M,\ \ I_{h_t}^0y=0.
\]
\par We consider the following symmetric three-level in $t$ method with a weight (parameter) $\sigma$ for the IBVP \eqref{hyperb2eq}-\eqref{hyperb2ibc} with $g=0$:
\begin{gather}
 B_h(\rho\Lambda_tv)+\sigma h_t^2A_h\Lambda_tv+A_hv=f\ \ \text{in}\ \ H_h\ \ \text{on}\ \ \omega_{h_t},
\label{3level sch1}\\[1mm]
 B_h(\rho\delta_tv^0)+\sigma h_t^2A_h\delta_tv^0+\half h_tA_hv^0=u_1+\half h_tf^0\ \ \text{in}\ \ H_h,
\label{3level sch2}
\end{gather}
where $v$: $\overline\omega_{h_t}\to H_h$ is the sought function and the functions $v^0,u_1\in H_h$ and $f$: $\{t_m\}_{m=0}^{M-1}\to H_h$ are given; we omit their dependence on $h$ for brevity.
Also $\sigma$ can depend on $\*h:=(h,h_t)$.

\par Note that the form of equation \eqref{3level sch2} for $v^1$ goes back to \cite{Z94} and is essential for several purposes (in particular, it seems most natural in the non-smooth case).
It can be rewritten in the form
{similar} to \eqref{3level sch1}:
\[
\tfrac{1}{0.5h_t}\big[B_h(\rho\delta_tv^0)+\sigma h_t^2A_h\delta_tv^0-u_1\big]+A_hv^0=f^0.
\]
{Note} that clearly $\tfrac{1}{0.5h_t}(\delta_tu^0-(\partial_tu)_{t=0})\approx(\partial_t^2u)_{t=0}$ for any function $u\in C^2[0,T]$.

\par Recall that linear algebraic systems in $H_h$ of the form
\begin{gather}
 B_h(\rho w^m)+\sigma h_t^2A_hw^m=b^m
\label{laeq}
\end{gather}
has to be solved at time levels $t_m$ to find the solution $v^{m+1}$ for all $0\leq m\leq M-1$.
One of the possible ways is to find directly $w^0=\delta_tv^0$ from \eqref{3level sch2}
and set $v^1=v^0+h_tw^0$,  {then  find}
$w=\Lambda_tv$ from \eqref{3level sch1} and set $\hat{v}=2v-\check{v}+h_t^2w$.
We can define the ``diagonal'' operator $D_\rho w:=\rho w$ in $H_h$, then $B_hD_\rho+\sigma h_t^2A_h$ is the operator in the problem \eqref{laeq}.

\par In \cite{BTT18,STT19}, for $\sigma\neq 0$, a special trick was applied (here we do not dwell on its motivation).
The auxiliary function $b$ is introduced by the recurrent relation
\begin{gather}
 \hat{b}=\big(2-\tfrac{1}{\sigma}\big)b-\check{b}-\tfrac{\rho}{\sigma^2h_t^2}v-\tfrac{1}{\sigma}\tilde{f}\ \ \text{in}\ \ H_h\ \ \text{on}\ \ \omega_{h_t},
\label{3termrrb}
\end{gather}
and it is suggested to solve the equation
\begin{gather}
 -A_h\hat{v}-\tfrac{1}{\sigma h_t^2}B_h(\rho\hat{v})=B_h\hat{b}\ \ \text{in}\ \ H_h\ \ \text{on}\ \ \omega_{h_t}
\label{tsynkov2}
\end{gather}
to find $\hat{v}$ (here the notation is slightly changed).
Rewriting relation \eqref{3termrrb} as
\[
 h_t^2\Lambda_tb=-\tfrac{1}{\sigma}b-\tfrac{\rho}{\sigma^2h_t^2}v-\tfrac{1}{\sigma}\tilde{f}
\]
and applying $-\sigma B_h$ to it, we get
\[
 -\sigma h_t^2\Lambda_t B_hb=B_hb+\tfrac{1}{\sigma h_t^2}B_h(\rho v)+B_h\tilde{f}.
\]
Applying to it (on both sides) equation \eqref{tsynkov2} from right to left (we use it also for $t_0=0$ together with $-A_hv^0-\tfrac{1}{\sigma h_t^2}B_h(\rho v^0)=B_hb^0$ for the definitions of $b^1$ and $b^0$),
we obtain
\[
 B_h(\rho\Lambda_tv)+\sigma h_t^2A_h\Lambda_tv=-A_hv-\tfrac{1}{\sigma h_t^2}B_h(\rho v)+\tfrac{1}{\sigma h_t^2}B_h(\rho v)+B_h\tilde{f}
 =-A_hv+B_h\tilde{f}
\]
on $\omega_{h_t}$.
This is nothing more than equation \eqref{3level sch1} with $f=B_h\tilde{f}$.

\par For $\sigma<\frac14$, we also assume that $A_h$ and $B_h$ are related by the following inequality
\begin{gather}
\|w\|_{A_h}\leq\alpha_h\|w\|_{B_h}\ \ \ \forall w\in H_h \ \ \Leftrightarrow\ \ A_h\leq\alpha_h^2 B_h.
\label{ahbh}
\end{gather}
Clearly the minimal value of $\alpha_h^2$ is the maximal eigenvalue of the generalized eigenvalue problem
\begin{gather}
A_he=\lambda B_he,\ \ e\in H_h,\ \ e\neq 0.
\label{geneigval}
\end{gather}
\par For method \eqref{3level sch1}-\eqref{3level sch2}, we present a theorem on uniform in time stability (unconditional for
$\sigma\geq \tfrac14$ or conditional for $\sigma<\tfrac14$) in the mesh strong (standard) and weak energy norms with respect to the initial data $v^0$ and $u_1$ and the free term $f$.
Let $\|y\|_{L_{h_t}^1(H_h)}:=\tfrac14 h_t\|y^0\|_h+I_{h_t}^{M-1}\|y\|_h$.
\begin{theorem}
\label{theo:1}
Let the operators $A_h$ and $B_h$ commute, i.e. $A_hB_h=B_hA_h$.
Let either $\sigma\geq\frac14$ and $\ve_0=1$, or
\begin{gather}
 \sigma<\tfrac14,\ \
 (\tfrac14-\sigma)h_t^2\alpha_h^2\leq (1-\ve_0^2)\urho\ \ \text{for some}\ \ 0<\ve_0<1.
\label{stabcond}
\end{gather}
For the solution to method \eqref{3level sch1}-\eqref{3level sch2}, the following bounds hold:
\par (1) in the strong energy norm
\begin{gather}
\max_{1\leq m\leq M}
\Big[\|\sqrt{\rho}\bar{\delta}_tv^m\|_h^2+(\sigma-\tfrac14)h_t^2\|\bar{\delta}_tv^m\|_{B_h^{-1}A_h}^2
+\|\bar{s}_tv^m\|_{B_h^{-1}A_h}^2\Big]^{1/2}
\nonumber\\[1mm]
\leq\Big(\|v^0\|_{B_h^{-1}A_h}^2+\ve_0^{-2}\big\|\tfrac{1}{\sqrt{\rho}}B_h^{-1}u_1\big\|_h^2\Big)^{1/2}
       +2\ve_0^{-1}\big\|\tfrac{1}{\sqrt{\rho}}B_h^{-1}f\big\|_{L_{h_t}^1(H_h)},
\label{energy est1}
\end{gather}
where the $f$-term can be replaced with
$2I_{h_t}^{M-1}\|(A_hB_h)^{-1/2}\bar{\delta}_tf\|_h+3\max\limits_{0\leq m\leq M-1}\|(A_hB_h)^{-1/2}f^m\|_h$;

\par (2) in the weak energy norm
\begin{gather}
\max_{0\leq m\leq M}
\max\Big\{\Big[\|\sqrt{\rho}v^m\|_h^2+(\sigma-\tfrac14)h_t^2\|v^m\|_{B_h^{-1}A_h}^2\Big]^{1/2},\,\|I_{h_t}^m\bar{s}_tv\|_{B_h^{-1}A_h}\Big\}
\nonumber\\[1mm]
 \leq\Big[\|\sqrt{\rho}v^0\|_h^2+(\sigma-\tfrac14)h_t^2\|v^0\|_{B_h^{-1}A_h}^2\Big]^{1/2}
 +2\|(A_hB_h)^{-1/2}u_1\|_h
 +2\|(A_hB_h)^{-1/2}f\|_{L_{h_t}^1(H_h)},
\label{energy est2}
\end{gather}
where, for $f=\delta_tg$, one can replace the $f$-term with $2\ve_0^{-1}I_{h_t}^{M}\big\|\tfrac{1}{\sqrt{\rho}}B_h^{-1}\big(g-s_tg^0\big)\big\|_h$.
\end{theorem}
\begin{proof}
We prove this theorem {by applying its non-standard reduction} to
the recently proved results directly suitable only for $\rho(x)\equiv{\rm const}$.
Applying $B_h^{-1}$ to equations \eqref{3level sch1}-\eqref{3level sch2}, we get
\begin{gather}
 (\rho I+\sigma h_t^2B_h^{-1}A_h)\Lambda_tv+B_h^{-1}A_hv=B_h^{-1}f\ \ \text{in}\ \ H_h\ \ \text{on}\ \ \omega_{h_t},
\label{3level sch1 -1}\\[1mm]
(\rho I+\sigma h_t^2B_h^{-1}A_h)\delta_tv^0+\half h_tB_h^{-1}A_hv^0=B_h^{-1}u_1+\half h_tB_h^{-1}f^0\ \ \text{in}\ \ H_h.
\label{3level sch2 -1}
\end{gather}
We have $D_\rho^*=D_\rho>0$.
The property $A_hB_h=B_hA_h$ implies $A_hB_h^{-1}=B_h^{-1}A_h$ and thus $(B_h^{-1}A_h)^*=B_h^{-1}A_h$.
Also the eigenvalue equation in \eqref{geneigval} can be rewritten as $B_h^{-1}A_he=\lambda e$ and therefore $B_h^{-1}A_h>0$;
moreover, inequality \eqref{ahbh} is equivalent to
\begin{gather}
 \|w\|_{B_h^{-1}A_h}\leq\alpha_h\|w\|_h\ \ \forall w\in H_h.
\label{ahbh -1}
\end{gather}
Consequently under the imposed conditions on $\sigma$ and $h_t$ we also have
\begin{gather}
 \ve_0^2\|\sqrt{\rho}w\|_h^2\leq\|\sqrt{\rho}w\|_h^2+(\sigma-\tfrac14)h_t^2\|w\|_{B_h^{-1}A_h}^2\ \ \forall w\in H_h.
\end{gather}
Now one can apply \cite[Theorem 1]{ZK20} (see also \cite{ZCh20}) concerning method \eqref{3level sch1}-\eqref{3level sch2} with $\rho(x)\equiv 1$ to method \eqref{3level sch1 -1}-\eqref{3level sch2 -1}, with
$D_\rho$, $B_h^{-1}A_h$, $B_h^{-1}f$ and $B_h^{-1}u_1$ in the role of $B_h$, $A_h$, $f$ and $u_1$, respectively, and derive the stated bounds.
\par We notice that \textit{the discrete energy conservation law}
\begin{gather}
 \|\sqrt{\rho}\bar{\delta}_tv^m\|_h^2+(\sigma-\tfrac14)h_t^2\|\bar{\delta}_tv^m\|_{B_h^{-1}A_h}^2
+\|\bar{s}_tv^m\|_{B_h^{-1}A_h}^2
 =\big(B_h^{-1}A_hv^0,s_tv^0\big)_h
\nonumber\\[1mm]
 +\big(B_h^{-1}u_1,\delta_tv^0\big)_h
 +\half h_t\big(B_h^{-1}f^0,\delta_tv^0\big)_h+2I_{h_t}^{m-1}(B_h^{-1}f,\mathring{\delta}_tv)_h, 1\leq m \leq M,
\label{en cons law}
\end{gather}
see proof of Theorem 1 in \cite{ZCh20},
{not only}
implies bound \eqref{energy est1} for the method
\eqref{3level sch1 -1}-\eqref{3level sch2 -1}  but itself has the independent interest.
This natural form is obtained, in particular, due to equation \eqref{3level sch2} for $v^1$.

\par Concerning the operators in the second form of the $f$-term in \eqref{energy est1} and the last two terms in \eqref{energy est2}, we also take into account the following transformations
\begin{gather*}
 \|(B_h^{-1}A_h)^{-1/2}B_h^{-1}w\|_h^2
 =\big((B_h^{-1}A_h)^{-1}B_h^{-1}w,B_h^{-1}w\big)_h
 =\big(B_h^{-1}A_h^{-1}w,w\big)_h
\\[1mm]
 =\big((A_hB_h)^{-1}w,w\big)_h
 =\|(A_hB_h)^{-1/2}w\|_h\ \ \forall w\in H_h.
\end{gather*}
This completes the proof.
\end{proof}

\par In practice, stability conditions like \eqref{stabcond} are often obtained by
applying the spectral method in the case $\rho(x)\equiv \textrm{const}$ and then taking in the result $\urho$ as this constant.
We emphasize that Theorem \ref{theo:1} justifies that such an approach is correct in our case.
On the other hand, we emphasize that in the case $\rho(x)\equiv \textrm{const}$ our stability bounds themselves
\eqref{energy est1}-\eqref{energy est2} \textit{differ} from those in \cite{ZK20}.
\par Recall that each of bounds \eqref{energy est1} or \eqref{energy est2} implies existence and uniqueness of the solution to method \eqref{3level sch1}-\eqref{3level sch2} for any given $v^0,u_1\in H_h$ and $f$: $\{t_m\}_{m=0}^{M-1}\to H_h$.
The same applies to finite-difference schemes below.
\par Bound \eqref{energy est2} in the weak energy norm is less standard
 than \eqref{energy est1} but namely it contains simple $H_h$-norm of $v^0$ most
relevant when studying stability with respect to the round-off errors; also  bounds in both norms are essential when proving delicate error estimates \cite{Z94} in dependence with the data smoothness.

\section{\large Construction and properties of compact finite-difference schemes of the 4th order of approximation}
\label{numerovschemes}
\setcounter{equation}{0}
\setcounter{lemma}{0}
\setcounter{theorem}{0}
\setcounter{remark}{0}

Let below $\Omega=(0,X_1)\times\ldots\times(0,X_n)$ and $g$ be general in the boundary condition \eqref{hyperb2ibc} if the opposite is not stated explicitly.
Define the uniform rectangular mesh
\[
 \bar{\omega}_h=\{x_{\*k}=(x_{1k_1},\ldots,x_{nk_n})=(k_1h_1,\ldots,k_nh_n);\, 0\leq k_1\leq N_1,\ldots,0\leq k_n\leq N_n\}
\]
in $\bar{\Omega}$
with the steps $h_1=\frac{X_1}{N_1},\ldots,h_n=\frac{X_n}{N_n}$, $h=(h_1,\ldots,h_n)$ and $\*k=(k_1,\ldots,k_n)$.
Let
\[
\omega_h=\{x_{\*k};\, 1\leq k_1\leq N_1-1,\ldots,1\leq k_n\leq N_n-1\},\ \ \partial\omega_h=\bar{\omega}_h\backslash\omega_h
\]
be the internal part and boundary of $\bar{\omega}_h$.
Define also the meshes $\omega_{\*h}:=\omega_h\times\omega_{h_t}$ in $Q_T$ and $\partial\omega_{\*h}=\partial\omega_h\times\{t_m\}_{m=1}^M$ on $\bar{\Gamma}_T$.

\par We introduce the standard difference approximation to $\partial_l^2w$:
\[
 (\Lambda_lw)_{\*k}=\tfrac{1}{h_l^2}(w_{\*k+\*e_l}-2w_{\*k}+w_{\*k-\*e_l}),\ \ l=1,\ldots,n,
\]
on $\omega_{h}$,
where $w_{\*k}=w(x_{\*k})$ and $\*e_1,\ldots,\*e_n$ is the standard coordinate basis in $\mathbb{R}^n$.
\par Let below $H_h$ be the space of functions defined on $\bar{\omega}_h$ and equal 0 on $\partial\omega_h$, endowed with the inner product $(v,w)_h=h_1\ldots h_n\sum_{x_{\*k}\in\omega_h}v_{\*k}w_{\*k}$.

\par We define the Numerov-type averaging operators and approximation of $f$
\begin{gather*}
 s_N:=I+\onetwelve (h_1^2\Lambda_1+\ldots+h_n^2\Lambda_n),\ \
 s_{N\widehat{j}}:=I+\onetwelve \sum_{1\leq i\leq n,\,i\neq j}h_i^2\Lambda_i,\ \ 1\leq j \leq n,
\\[1mm]
 A_N:=-(a_1^2s_{N\widehat{1}}\Lambda_1+\ldots+a_n^2s_{N\widehat{n}}\Lambda_n),\ \
 f_N:=s_Nf+\onetwelve h_t^2\Lambda_tf,
\end{gather*}
where $I$ is the identity operator;
note that $s_{N\widehat{1}}=I$ for $n=1$.
We also set
\begin{gather}
 u_{1N}:=s_N(\rho u_1) +\onetwelve h_t^2(a_1^2\Lambda_1+\ldots+a_n^2\Lambda_n)u_1,\,\
\label{u1N}\\[1mm]
 f_N^0:=f_{dh_t}^{(0)}+\onetwelve (h_1^2\Lambda_1+\ldots+h_n^2\Lambda_n)f_0,
 \ \ \text{with some}\ \
 f_{dh_t}^{(0)}=f_{d}^{(0)}+\mathcal{O}(h_t^3),
\label{tf0N}
\end{gather}
on $\omega_h$, where $f_{d}^{(0)}:=f_0+\tfrac13 h_t(\partial_tf)_0+\onetwelve h_t^2(\partial_t^2f)_0$
and $y_0:=y|_{t=0}$, similarly to \cite{ZK20}.
Note the non-trivial form of $u_{1N}\approx \rho u_1$, where the first term contains $\rho u_1$, but the second one does not.
Additional details concerning formula \eqref{tf0N} are given in Remark \ref{rem:fdht0} below.
\par The following basic lemma generalizes
\cite[Lemmas 1-2]{ZK20} for $\rho(x)\not\equiv{\rm const}$.
\begin{lemma}
\label{lem:psi}
Let the coefficient $\rho$ and solution $u$ to the IBVP \eqref{hyperb2eq}-\eqref{hyperb2ibc} be sufficiently smooth respectively in $\bar{\Omega}$ and $\bar{Q}_T$.
Then the following formulas hold
\begin{gather}
 s_N(\rho\Lambda_tu)-\onetwelve h_t^2(h_1^2a_1^2\Lambda_1+\ldots+a_n^2h_n^2\Lambda_n)\Lambda_tu-A_Nu-f_N
 =\mathcal{O}(|\*h|^4)\,\
 \text{on}\,\ \omega_{\*h},
\label{approxerrN}\\[1mm]
 s_N(\rho\delta_tu)^0-\tfrac{h_t^2}{12}(h_1^2a_1^2\Lambda_1+\ldots+a_n^2h_n^2\Lambda_n)(\delta_tu)^0
-\tfrac{h_t}{2}A_Nu_0
 -u_{1N}-\tfrac{h_t}{2}f_N^0=\mathcal{O}(|\*h|^4)\
 \text{on}\ \omega_h.
\label{psie0}
\end{gather}
\end{lemma}
\begin{proof}
We apply a new technique based on averaging of equation \eqref{hyperb2eq} related to the polylinear finite elements like in \cite{ZK20};
the more standard Numerov-type technique could be also used.
An advantage of the  averaging technique is that approximations of $f$ and $u_1$ in the non-smooth case (important in practice) become clear from the right-hand sides of formulas \eqref{avereq} and \eqref{avereq_0} below (we use this in Section \ref{numerexperiments}) but remain obscure in the frame of the Numerov-type technique.
\par 1. We define the well-known average in the variable $x_k$ related to the linear finite elements
\[
 (q_kw)(x_k)=\tfrac{1}{h_k}\int_{-h_k}^{h_k}w(x_k+\xi)\big(1-\tfrac{|\xi|}{h_k}\big)\,d\xi.
\]
For a function $w(x_k)$ smooth on $[0,X_k]$, the following relations hold
\begin{gather}
 q_k\partial_k^2w=\Lambda_kw,\ \
\label{qklambdak}\\[1mm]
 q_kw=w+q_k\rho_{k2}(\partial_k^2w),
\label{qkw 2}\\[1mm]
 q_kw=w+\onetwelve h_k^2\partial_k^2w+q_k\rho_{k4}(\partial_k^4w)=w+\onetwelve h_k^2\Lambda_kw+\tilde{\rho}_{k4}(\partial_k^4w),
\label{qkw 3}\\[1mm]
 |q_k\rho_{ks}(\partial_k^sw)|\leq c_sh_k^s\|\partial_k^sw\|_{C(I_{kl})},\, s=2,4,\
 |\tilde{\rho}_{k4}(\partial_k^4w)|\leq \tilde{c}_4h_k^4\|\partial_k^4w\|_{C(I_{kl})}
\label{resid_bound}
\end{gather}
at the nodes $x_k=x_{kl}$, $1\leq l\leq N_k-1$, with the constants $c_s$ and $\tilde{c}_4$ independent on the mesh and $I_{kl}:=[x_{k(l-1)},x_{k(l+1)}]$.
Formula \eqref{qklambdak} is well-known and is checked by integrating by parts.
Other relations hold due to Taylor's formula at $x_k=x_{kl}$ with the residual in the integral form
\begin{gather}
 \rho_{ks}(w)(x_k)=\frac{1}{(s-1)!}\int_{x_{kl}}^{x_k}w(\xi)(x_k-\xi)^{s-1}\,d\xi,
\label{taylor_residual}
\end{gather}
for $s=2,4$,
together with the elementary formula
\[
 \tfrac{1}{h_k}\int_{-h_k}^{h_k}\half \xi^2\big(1-\tfrac{|\xi|}{h_k}\big)\,d\xi=\onetwelve h_k^2.
\]
The respective formulas hold for the averaging operator $q_t$ in the variable $t=x_{n+1}$ as well since one can set $X_{n+1}=T$ and $h_{n+1}=h_t$.

\par We apply the operator $\bar{q}q_t$ with $\bar{q}:=q_1\ldots q_n$ to the wave equation
\eqref{hyperb2eq} at the nodes of $\omega_{\*h}$, use formula \eqref{qklambdak} and get:
\begin{gather}
 \bar{q}(\rho\Lambda_t u)-(a_1^2\bar{q}_{\,\widehat{1}}\Lambda_1u+\ldots+a_n^2\bar{q}_{\,\widehat{n}}\Lambda_nu)
 =\bar{q}q_tf,\ \ \text{with}\ \
 \bar{q}_{\,\widehat{i}}:=\prod_{1\leq k\leq n,\, k\neq i}q_k;
\label{avereq}
\end{gather}
here $\bar{q}_{\,\widehat{1}}=I$ for $n=1$.
The above expansions for $q_1,\ldots,q_n,q_{n+1}=q_t$ lead to the formula
\begin{gather*}
 \rho\Lambda_tu+\sum_{i=1}^n\onetwelve h_i^2\Lambda_i(\rho\Lambda_tu)
 -\sum_{i=1}^n a_i^2\Big[\Lambda_iu+\Big(\sum_{1\leq j\leq n,\, j\neq i}\onetwelve h_j^2\Lambda_j\Big)\Lambda_iu
 +\onetwelve h_t^2\Lambda_t\Lambda_i u\Big]
\\[1mm]
 =f+\onetwelve \sum_{i=1}^n h_i^2\Lambda_if
 +\onetwelve h_t^2\Lambda_tf+O(|\*h|^4),
\end{gather*}
and thus, using the above defined operators $s_N$ and $A_N$ as well as $f_N$, to formula \eqref{approxerrN} as well.

\par 2. In addition, we define the one-sided average in $t$ over $(0,h_t)$:
\begin{gather}
 q_ty^0=\tfrac{2}{h_t}\int_0^{h_t}y(t)\big(1-\tfrac{t}{h_t}\big)\,dt.
\label{qty0}
\end{gather}
We apply $\tfrac{h_t}{2}\bar{q}q_t(\cdot)^0$ to the wave equation \eqref{hyperb2eq} and, since
$\tfrac{h_t}{2}(q_t\partial_tu)^0=(\delta_tu)^0-(\partial_tu)_0$,  obtain
\begin{gather}
 \bar{q}(\rho\delta_tu)^0
 -\tfrac{h_t}{2}(a_1^2\bar{q}_{\,\widehat{1}}\Lambda_1u+\ldots+a_n^2\bar{q}_{\,\widehat{n}}\Lambda_nu)q_tu^0
 =\bar{q}(\rho u_1)+\tfrac{h_t}{2}\bar{q}q_tf^0.
\label{avereq_0}
\end{gather}
Using Taylor's formula at $t=0$ and calculating the arising integrals in $t$ over $(0,h_t)$, we get
\begin{gather}
 \tfrac{h_t}{2}q_tf^0=\tfrac{h_t}{2}f_0+\tfrac{h_t^2}{6}(\partial_tf)_0+\tfrac{h_t^3}{24}(\partial_t^2f)_0+O(h_t^4)
 =\tfrac{h_t}{2}f_d^{(0)}+O(h_t^4),
\label{expan_f}
\end{gather}
with $f_d^{(0)}$ defined above.
Here we omit the integral representations for $O(h_t^4)$-terms for brevity.
Similarly to the previous Item 1 and due to expansion \eqref{expan_f}, we find
\begin{gather}
 \bar{q}(\rho\delta_tu)^0=s_N(\rho\delta_tu)^0+O(|h|^4),\ \
 \bar{q}(\rho u_1)=s_N(\rho u_1)+O(|h|^4),
\label{barq_prop}\\[1mm]
 \tfrac{h_t}{2}q_t\bar{q}f^0=\tfrac{h_t}{2}f_d^{(0)}+\tfrac{1}{12}h_i^2\Lambda_if_0+O(|\*h|^4).
\label{aver_prop}
\end{gather}
\par Also due to Taylor's formula in $t$ at $t=0$ one can write down
\[
 u(\cdot,t)=u_0+tu_1+\tfrac{t^2}{h_t}((\delta_tu)^0-u_1)+O(t^3).
\]
Thus similarly
first to \eqref{expan_f} and second to
\eqref{barq_prop} as well as according to formula \eqref{qkw 2} and the first bound \eqref{resid_bound} we obtain
\begin{gather}
 \tfrac{h_t}{2}a_k^2\bar{q}_{\,\widehat{k}}\Lambda_kq_tu^0
 =a_k^2\big[\tfrac{h_t}{2}\bar{q}_{\,\widehat{k}}\Lambda_ku_0+\tfrac{h_t^2}{6}\bar{q}_{\,\widehat{k}}\Lambda_ku_1
 +\tfrac{h_t^2}{12}\bar{q}_{\,\widehat{i}}\Lambda_k((\delta_tu)^0-u_1)\big]+O(h_t^4)
\nonumber\\[1mm]
 =\tfrac{h_t}{2}a_k^2s_{N\widehat{k}}\Lambda_ku_0
 +\tfrac{h_t^2}{12}a_k^2\Lambda_ku_1
 +\tfrac{h_t^2}{12}a_k^2s_{N\widehat{k}}\Lambda_k(\delta_tu)^0+O(|\*h|^4),\ \ 1\leq k\leq n.
\label{aver_equat 0}
\end{gather}
We insert all the derived expansions \eqref{expan_f}-\eqref{aver_equat 0} into formula \eqref{avereq_0}, rearrange the summands and
obtain formula \eqref{psie0} with $u_{1N}$ and $f_N^0$ defined above.
\end{proof}

\begin{remark}
\label{rem:fdht0}
Let $0<h_t\leq\bar{h}_t\leq T$.
If $f$ is sufficiently smooth in $t$ in $\bar{Q}_{\bar{h}_t}$ (or $\bar{\Omega}\times [-\bar{h}_t,\bar{h}_t]$), then $f_{dh_t}^{(0)}=f_{d}^{(0)}+\mathcal{O}(h_t^3)$ (see \eqref{tf0N}) for the following three- and two-level approximations
\begin{gather}
{f}_{dh_t}^{(0)}=\tfrac{7}{12}f^0+\half f^1-\onetwelve f^2,\ \
{f}_{dh_t}^{(0)}=\tfrac13f^0+\tfrac23f^{1/2}\ \ \text{with}\ \  f^{1/2}:=f|_{t=h_t/2}
\label{ftd02}
\end{gather}
(or
$f_{dh_t}^{(0)}=f^0+\tfrac13 h_t\mathring{\delta}_tf^0+\onetwelve h_t^2\Lambda_tf^0
 =-\onetwelve f^{-1}+\tfrac56f^0+\tfrac14 f^1$ with $f^{-1}:=f|_{t=-h_t}$).
These formulas are easily checked using Taylor's formula at $t=0$.
\end{remark}

\par In our construction of compact schemes for the IBVP \eqref{hyperb2eq}-\eqref{hyperb2ibc}, in general we will follow \cite{ZK20}.
Preliminarily we consider the scheme of the form
\begin{gather}
 s_N(\rho\Lambda_tv)-\onetwelve h_t^2(a_1^2\Lambda_1+\ldots+a_n^2\Lambda_n)\Lambda_tv
 +A_Nv =f_N\ \ \text{on}\ \ \omega_{\*h},
\label{num0eq}\\[1mm]
  v|_{\partial\omega_{\*h}}=g,\ \
  s_N(\rho\delta_tv^0)-\onetwelve h_t^2(a_1^2\Lambda_1+\ldots+a_n^2\Lambda_n)v^0
  +\half h_tA_Nv^0
  =u_{1N}+\half h_tf_N^0\ \ \text{on}\ \ \omega_h.
\label{num0ic}
\end{gather}
On the left in formulas \eqref{approxerrN}-\eqref{psie0} in Lemma \ref{lem:psi}, the approximation errors of the equations for this scheme stand, and thus these formulas mean here that the scheme has the approximation order $\mathcal{O}(|\*h|^4)$.
For $n=1$, the scheme takes the simplest form
\begin{gather}
 s_N(\rho\Lambda_tv)-\onetwelve h_t^2a_1^2\Lambda_1\Lambda_tv-a_1^2\Lambda_1v=f_N\ \ \text{on}\ \ \omega_{\*h},
\label{num0eq 1d}\\[1mm]
  v|_{\partial\omega_{\*h}}=g,\ \
  s_N(\rho\delta_tv^0)-\onetwelve h_t^2a_1^2\Lambda_1\delta_tv^0-\half h_ta_1^2\Lambda_1v^0=u_{1N}+\half h_tf_N^0\ \ \text{on}\ \ \omega_h,
\label{num0ic 1d}
\end{gather}
that is a particular case (for $g=0$) of the general method \eqref{3level sch1}-\eqref{3level sch2} for $B_h=s_N$, $A_h=-a_1^2\Lambda_1$ and $\sigma=\onetwelve$.

\par But for $n\geq 2$ scheme \eqref{num0eq}-\eqref{num0ic} is no more of type \eqref{3level sch1}-\eqref{3level sch2}.
Therefore we first replace it with the following scheme
\begin{gather}
 s_N(\rho\Lambda_tv)+\onetwelve h_t^2A_N\Lambda_tv+A_Nv=f_N\ \ \text{on}\ \ \omega_{\*h},
\label{num1eq}\\[1mm]
  v|_{\partial\omega_{\*h}}=g,\ \
  s_N(\rho\delta_tv^0)+\onetwelve h_t^2A_N\delta_tv^0+\half h_tA_Nv^0=u_{1N}+\half h_tf_N^0\ \ \text{on}\ \ \omega_h,
\label{num1ic}
\end{gather}
that corresponds to the case $B_h=s_N$, $A_h=A_N$ and $\sigma=\onetwelve$.
Since
\[
 A_N+a_1^2\Lambda_1+\ldots+a_n^2\Lambda_n=a_1^2(I-s_{N\widehat{1}})\Lambda_1+\ldots+a_n^2(I-s_{N\widehat{n}})\Lambda_n,
\]
the approximation error of this scheme is also of the order $\mathcal{O}(|\*h|^4)$.

\par For $n=2$, one can easily generalize this scheme by the extension
\begin{gather}
 s_N=I+\tfrac{1}{12}h_1^2\Lambda_1+\tfrac{1}{12}h_2^2\Lambda_2 \mapsto s_{N\beta}:=s_N+\beta\tfrac{h_1^2}{12}\tfrac{h_2^2}{12}\Lambda_1\Lambda_2,
\label{sNbeta}
\end{gather}
with the parameter $\beta$, keeping its approximation order.
Note that $\Lambda_1\Lambda_2>0$ in $H_h$.

\par But the last scheme fails for $n\geq 3$ similarly to \cite{DZR15,ZK20}.
Recall that the point is that the minimal eigenvalue of $s_N$ as the operator in $H_h$ is such that
\[
 \lambda_{\min}(s_N)>1-\tfrac{n}{3},\ \ \lambda_{\min}(s_N)=1-\tfrac{n}{3}+O\big(\tfrac{1}{N_1^2}+\ldots+\tfrac{1}{N_n^2}\big)
\]
that is suitable only for $n=1,2$,
since $s_N$ becomes almost singular for $n=3$ and even $\lambda_{\min}(s_N)<0$ for $n\geq 4$, for small $|h|$ (and a crucial property $s_N>0$ is not valid any more).
Thus for $n=3$ it is of sense to replace $s_N$ with $\bar{s}_N$ and pass to the scheme
\begin{gather}
 \bar{s}_N(\rho\Lambda_tv)+\onetwelve h_t^2 A_N\Lambda_tv+A_Nv=f_N\ \ \text{on}\ \ \omega_{\*h},
\label{num2eq}\\[1mm]
 v|_{\partial\omega_{\*h}}=g,\ \
 \bar{s}_N(\rho\delta_tv^0)+\onetwelve h_t^2A_N\delta_tv^0+\half h_tA_Nv^0=u_{1N}+\half h_tf_N^0\ \ \text{on}\ \ \omega_h.
\label{num2ic}
\end{gather}
Next, for any $n\geq 1$, one can further replace $A_N$ with $\bar{A}_N$ and get the following unified scheme
\begin{gather}
 \bar{s}_N(\rho\Lambda_tv)+\onetwelve h_t^2\bar{A}_N\Lambda_tv+\bar{A}_Nv=f_N\ \ \text{on}\ \ \omega_{\*h},
\label{num3eq}\\[1mm]
 v|_{\partial\omega_{\*h}}=g,\ \
 \bar{s}_N(\rho\delta_tv^0)+\onetwelve h_t^2\bar{A}_N\delta_tv^0+\half h_t\bar{A}_Nv^0=u_{1N}+\half h_tf_N^0\ \ \text{on}\ \ \omega_h
\label{num3ic}
\end{gather}
(for $\rho(x)\equiv 1$, it goes back to \cite{DZR15} in the case of the time-dependent Schr\"{o}dinger equation).
In the last two schemes, we use the operators
\begin{gather}
 \bar{s}_N:=\prod_{k=1}^ns_{kN},\,\ \bar{s}_{N\widehat{l}}:=\prod_{1\leq k\leq n,\,k\neq l}s_{kN},\,\
 s_{kN}:=I+\onetwelve h_k^2\Lambda_k,\,\
\label{bsN}\\[1mm]
 \bar{A}_N:=-(a_1^2\bar{s}_{N\widehat{1}}\Lambda_1+\ldots+a_n^2\bar{s}_{N\widehat{n}}\Lambda_n),
\label{bAN}
\end{gather}
where $\bar{s}_N$ is the splitting version of $s_N$, and $\bar{s}_{N\widehat{l}}$ is similar to
$\bar{s}_N$ excluding the direction $x_l$,
with $\bar{s}_{N\widehat{1}}=I$ for $n=1$.
All of them are symmetric positive definite as the operators in $H_h$.

\par We also have $(\tfrac23)^nI<\bar{s}_N<I$ in $H_h$.
The following formula connects $\bar{s}_N$ and $s_N$
\begin{gather}
 \bar{s}_N=s_N+\sum_{k=2}^n\bar{s}_N^{\,(k)},\ \
 \bar{s}_N^{\,(k)}:=(\onetwelve)^k\sum_{1\leq i_1<\ldots<i_k\leq n}h_{i_1}^2\ldots h_{i_k}^2\Lambda_{i_1}\ldots\Lambda_{i_k}.
\label{bsNsN}
\end{gather}
Notice that $(-1)^k\bar{s}_N^{\,(k)}>0$ in $H_h$, $2\leq k\leq n$.

\par Here $\bar{A}_N=A_N$ for $n=1,2$, and for $n=1$ the last scheme coincides with \eqref{num0eq 1d}-\eqref{num0ic 1d} but
\begin{gather}
 \bar{A}_N=A_N+\bar{A}_N^{(3)}
 =-(a_1^2\Lambda_1+a_2^2\Lambda_2+a_3^2\Lambda_3)+\bar{A}_N^{(2)}+\bar{A}_N^{(3)},\ \
\nonumber\\[1mm]
 \bar{A}_N^{(2)}:=-\onetwelve\big[(a_1^2h_2^2+a_2^2h_1^2)\Lambda_1\Lambda_2
            +(a_1^2h_3^2+a_3^2h_1^2)\Lambda_1\Lambda_3
            +(a_2^2h_3^2+a_3^2h_2^2)\Lambda_2\Lambda_3\big],
\nonumber\\[1mm]
\bar{A}_N^{(3)}:=-\tfrac{1}{12^2}(a_1^2h_2^2h_3^2+a_2^2h_1^2h_3^2+a_3^2h_1^2h_2^2)\Lambda_1\Lambda_2\Lambda_3
\label{barA N}
\end{gather}
for $n=3$, with $\bar{A}_N^{(2)}<0$ and $\bar{A}_N^{(3)}>0$ in $H_h$.

\par Due to the formulas
\[
 \bar{A}_N-A_N=-a_1^2(\bar{s}_{N\widehat{1}}-s_{N\widehat{1}})\Lambda_1-\ldots-a_n^2(\bar{s}_{N\widehat{n}}-s_{N\widehat{n}})\Lambda_n
\]
and \eqref{bsNsN},
the approximation errors of schemes \eqref{num2eq}-\eqref{num2ic} and \eqref{num3eq}-\eqref{num3ic} have the same order $\mathcal{O}(|\*h|^4)$ as the preceding scheme \eqref{num1eq}-\eqref{num1ic}.

For $n=3$, one can easily generalize scheme \eqref{num2eq}-\eqref{num2ic} by the extensions
\begin{gather}
 \bar{s}_N=s_{1N}s_{2N}s_{3N} \mapsto
 s_{N\beta\gamma}:=s_N
 +\beta\bar{s}_N^{\,(2)}
 +\gamma\bar{s}_N^{\,(3)},\ \
 A_N \mapsto A_{N\theta}:=A_N+\theta\bar{A}_N^{(3)},
\label{ANtheta}
\end{gather}
with the three parameters $\beta,\gamma$ and $\theta$, keeping its approximation order.
Here we have explicitly
\begin{gather}
 \bar{s}_N^{\,(2)}=\tfrac{1}{12^2}(h_1^2h_2^2\Lambda_1\Lambda_2+h_1^2h_3^2\Lambda_1\Lambda_3+h_2^2h_3^2\Lambda_2\Lambda_3),\ \
 \bar{s}_N^{\,(3)}=\tfrac{1}{12^3}h_1^2h_2^2h_3^2\Lambda_1\Lambda_2\Lambda_3\ \ \text{for}\ \ n=3
\label{bar s N}
\end{gather}
as well as
\[
 s_{N\beta}=(1-\beta)s_N+\beta\bar{s}_N\ \text{for}\ n=2;\,\
 s_{N\beta\beta}=(1-\beta)s_N+\beta\bar{s}_N,\
 A_{N\theta}:=(1-\theta)A_N+\theta \bar{A}_N\ \text{for}\ n=3.
\]
\par The following explicit expansions in $\Lambda_k$ for the operators at the upper time level in \eqref{num1eq} for $n=2$ and \eqref{num2eq} for $n=3$ hold
\begin{gather*}
 s_N(\rho w)+\onetwelve h_t^2A_Nw=\rho w
 +\onetwelve\big[(h_1^2\Lambda_1+h_2^2\Lambda_2)(\rho w)-h_t^2(a_1^2\Lambda_1+a_2^2\Lambda_2)w\big]
\\[1mm]
 -(\onetwelve)^2h_t^2\big(a_1^2h_2^2+a_2^2h_1^2\big)\Lambda_1\Lambda_2w\ \ \text{for}\ \ n=2,
\\[1mm]
 \bar{s}_N(\rho w)+\onetwelve h_t^2\bar{A}_Nw
 =\rho w+\onetwelve\big[(h_1^2\Lambda_1+h_2^2\Lambda_2+h_3^2\Lambda_3)(\rho w)
 -h_t^2(a_1^2\Lambda_1+a_2^2\Lambda_2+a_3^2\Lambda_3)w\big],
\\[1mm]
 +\bar{s}_N^{\,(2)}(\rho w)
 -\tfrac{1}{12^2}h_t^2\big[(a_1^2h_2^2+a_2^2h_1^2)\Lambda_1\Lambda_2
                          +(a_1^2h_3^2+a_3^2h_1^2)\Lambda_1\Lambda_3
                          +(a_2^2h_3^2+a_3^2h_2^2)\Lambda_2\Lambda_3\big]w
\\[1mm]
 +\bar{s}_N^{\,(3)}(\rho w)-\onetwelve h_t^2\bar{A}_N^{(3)}w\ \ \text{for}\ \ n=3,
\end{gather*}
see also formulas in \eqref{bar s N} and \eqref{barA N} for the last two terms.
In the particular case of $a_i$ and $h_i$ independent on $i$ (i.e., for the square spatial mesh), the formulas are simplified, and the operators on the left in them differ only up to factors from those given in the related formulas (21)-(22) in \cite{BTT18} and (11) in \cite{STT19}.
Moreover, turning to formulas \eqref{3termrrb}-\eqref{tsynkov2}, one can show that in this case equations \eqref{num1eq} for $n=2$ and \eqref{num3eq} for $n=3$ are equivalent to respective methods from \cite{BTT18,STT19} up to our simpler approximations of $f$.
But it should be emphasized that
we prefer to supplement them by other than in \cite{BTT18,STT19} similar equations \eqref{num1ic} and \eqref{num3ic} for $v^1$.

\par Also, in the same particular case, the family of methods with the operators
\begin{gather}
 s_{N\beta\gamma}(\rho w)+\onetwelve h_t^2A_{N\theta},\ \
 \text{with}\ \ \beta=2,\ \ \gamma=12(1-\varkappa),\ \theta=4(\varkappa-1), \ -\tfrac12<\varkappa<3
\label{family tsynkov}
\end{gather}
at the upper level was also studied in \cite[Section 3.2.2]{STT19},
though according to the above analysis, the values $\varkappa\leq 1$ (including the so-called canonical based scheme for $\varkappa=1$ in \cite{STT19}) can hardly be recommended for exploiting.
These methods are related to equation \eqref{num2eq} with the extended operators \eqref{ANtheta} in the same way (actually for any $\beta$, $\gamma$ and~$\theta$).
\par Now we prove the conditional stability theorem for all the above constructed schemes.
\begin{theorem}
\label{theo:2}
Let $g=0$ in \eqref{hyperb2ibc}.
Let us consider:
\begin{enumerate}
\item scheme \eqref{num1eq}-\eqref{sNbeta} for $n=2$,
\item scheme \eqref{num2eq}-\eqref{num2ic} and \eqref{ANtheta} for $n=3$,
\item scheme \eqref{num3eq}-\eqref{num3ic} for $n\geq 1$ (for $n=1$, the scheme \eqref{num0eq 1d}-\eqref{num0ic 1d} is the same)
\end{enumerate}
and set respectively
$(B_h,A_h)=(s_{N\beta},A_N)$,
$(B_h,A_h)=(s_{N\beta\gamma},A_{N\theta})$ and
$(\bar{s}_N,\bar{A}_N)$.

\par Let the parameters $\beta,\gamma$ and $\theta$ be chosen such that $B_h>0$ and $A_h>0$ in $H_h$,
and $A_h\leq\alpha_h^2 B_h$ with some $\alpha_h$ (see \eqref{ahbh}) for the first and second schemes.
Let also $0<\ve_0<1$, and the condition
\begin{gather}
 \tfrac16 h_t^2\alpha_h^2\leq (1-\ve_0^2)\urho,
\label{stabcondN}
\end{gather}
for the first and second schemes, or the explicit condition
\begin{gather}
 h_t^2\big(\tfrac{a_1^2}{h_1^2}+\ldots+\tfrac{a_n^2}{h_n^2}\big)
 \leq (1-\ve_0^2)\urho
\label{condhth2bar}
\end{gather}
for the third scheme, be valid (see also Remark \ref{betgamthet} below).
Then, for any free terms $f_N$: $\{t_m\}_{m=0}^{M-1}\to H_h$ and $u_{1N}\in H_h$ (not only for those specific defined above),
the solutions to all three schemes satisfy the following two stability bounds:
\begin{gather}
\max_{1\leq m\leq M}\Big(\ve_0^2\|\sqrt{\rho}\bar{\delta}_tv^m\|_h^2+\|\bar{s}_tv^m\|_{B_h^{-1}A_h}^2\Big)^{1/2}
\nonumber\\[1mm]
\leq\Big(\|v^0\|_{B_h^{-1}A_h}^2+\ve_0^{-2}\big\|\tfrac{1}{\sqrt{\rho}}B_h^{-1}u_{1N}\big\|_h^2\Big)^{1/2}
       +2\ve_0^{-1}\big\|\tfrac{1}{\sqrt{\rho}}B_h^{-1}f_N\big\|_{L_{h_t}^1(H_h)},
\label{energy est1N}
\end{gather}
where the $f_N$-term can be taken also as
$2I_{h_t}^{M-1}\|B_h^{-1/2}A_h^{-1/2}\bar{\delta}_tf_N\|_h
 +3\max\limits_{0\leq m\leq M-1}\|B_h^{-1/2}A_h^{-1/2}f_N^m\|_h;$
\begin{gather}
 \max_{0\leq m\leq M}
 \max\big\{\ve_0\|\sqrt{\rho}v^m\|_h,\,\|I_{h_t}^m\bar{s}_tv\|_{B_h^{-1}A_h}\big\}
\nonumber\\[1mm]
 \leq\|\sqrt{\rho}v^0\|_h
 +2\|B_h^{-1/2}A_h^{-1/2}u_{1N}\|_h
 +2\|B_h^{-1/2}A_h^{-1/2}f_N\|_{L_{h_t}^1(H_h)},
\label{energy est2N}
\end{gather}
where, for $f_N=\delta_tg$, one can replace the $f_N$-term with $2\ve_0^{-1}I_{h_t}^{M}\big\|\tfrac{1}{\sqrt{\rho}}B_h^{-1}\big(g-s_tg^0\big)\big\|_h$.
\end{theorem}
\begin{remark}
\label{betgamthet}
Let us comment on the stability condition \eqref{stabcondN}.
For $(B_h,A_h)=(s_{N\beta},A_N)$ with $\beta\geq 0$ for $n=2$,
\[
 (B_h,A_h)=(s_{N\beta\gamma},A_{N\theta})\ \ \text{with}\ \ \beta\geq\ve_1,\ \gamma\leq\ve_1\ \ \text{with some}\ \ 0<\ve_1\leq 1,
 \ 0\leq\theta\leq 1\ \ \text{for}\ \ n=3
\]
and $(B_h,A_h)=(\bar{s}_N,\bar{A}_N)$ for $n\geq 1$, conditions $B_h>0$ and $A_h>0$ in $H_h$ hold, as well as condition \eqref{ahbh} has recently been studied in \cite[Lemma 3]{ZK20} (for $\beta=0$ and $\theta=0,1$ that is enough here).
Consequently condition \eqref{stabcondN} is valid under the assumption
\[
 C_1h_t^2\big(\tfrac{a_1^2}{h_1^2}+\ldots+\tfrac{a_n^2}{h_n^2}\big)\leq (1-\ve_0^2)\urho
\]
where $C_1=\frac43,\ve_1^{-1}$ or $1$
respectively for the first, second or third scheme.
The reason is that, under the assumptions made on $\beta$, $\gamma$ and $\theta$, the following operator inequalities in $H_h$ hold
\begin{gather}
 s_N \leq s_{N\beta}\ \ \text{for}\ \ n=2,\ \
 \ve_1\bar{s}_N\leq s_{N\beta\gamma}\ \text{and}\ A_{N\theta}\leq\bar{A}_N\ \ \text{for}\ \ n=3.
\label{ineq sNbet gam}
\end{gather}
This is an example, and we do not intend here to study condition \eqref{ahbh} for general $\beta$, $\gamma$ and $\theta$.
\end{remark}
\begin{proof}
The theorem follows directly from the general stability Theorem \ref{theo:1}, for $B_h$ and $A_h$ listed in the statement, in the particular case $\sigma=\onetwelve$, specifying assumption \eqref{stabcond} and inequality \eqref{ahbh -1}.
Here $B_h$ and $A_h$ commute since they have the same system of eigenvectors in $H_h$.

\par In the second form of the $f_N$-term in \eqref{energy est1N} and in the last two terms on the right in \eqref{energy est2N}, we also take into account that $(A_hB_h)^{-1/2}=B_h^{-1/2}A_h^{-1/2}$ due to the last mentioned property.
\end{proof}
\begin{remark}
\label{Bhnu}
Usually $\nu_0 I\leq B_h\leq\nu I$ in $H_h$ with some $\nu\geq\nu_0>0$ both independent of $\*h$;
in particular, under the assumptions on $\beta$ and $\gamma$ from Remark \ref{betgamthet} one has
\[
 \tfrac13 I<s_{N\beta}<\big(1+\tfrac{1}{9}\beta\big)I\ \ \text{for}\ \ n=2,\ \
 \ve_1(\tfrac23)^3I<s_{N\beta\gamma}<\big(1+\tfrac{1}{3}\beta+\tfrac{1}{27}\max\{-\gamma,0\}\big)I\ \ \text{for}\ \ n=3
\]
due to the
inequalities $-\frac14 h_k^2\Lambda_k<I$, \eqref{ineq sNbet gam} and $(\tfrac23)^nI<\bar{s}_N<I$ in $H_h$.
Then one can simplify the above stability bounds replacing the operator $B_h^{-1}$ with the constant $\nu^{-1}$ on the left and/or replacing $B_h^{-1}$ with $\nu_0^{-1}$ and $B_h^{-1/2}$ with $\nu_0^{-1/2}$ on the right.
\end{remark}
\par Next, based on Theorem \ref{theo:2}, we prove the 4th order error bound for the same schemes.
\begin{theorem}
\label{theo:3}
Let the coefficient $\rho$ and solution $u$ to the IBVP \eqref{hyperb2eq}-\eqref{hyperb2ibc} be sufficiently smooth respectively in $\bar{\Omega}$ and $\bar{Q}_T$.
Then under the hypotheses of Theorem \ref{theo:2} but excluding $g=0$ and for $\nu_0 I\leq B_h\leq\nu I$ with some $\nu\geq\nu_0>0$ (see Remark \ref{Bhnu}) as well as $v^0=u_0$ on $\bar{\omega}_h$, for all three schemes listed in it, the following 4th order error bound in the strong energy norm holds
\[
 \max_{1\leq m\leq M}\big[\ve_0\|\sqrt{\rho}\bar{\delta}_t(u-v)^m\|_h+\|\bar{s}_t(u-v)^m\|_{A_h}\big]=\mathcal{O}(|\*h|^4).
\]
\par Let $a_{\min}=\min_{1\leq i\leq n}a_i$, $\Delta_h=\Lambda_1+\ldots+\Lambda_n$ be the simplest approximation of the Laplace operator, and $\ve_2=1$ for the first and third schemes or $0<\ve_2\leq\theta\leq 1$ for the second one.
Then
\begin{equation}
 \sqrt{\ve_2}a_{\min}(\tfrac{2}{3})^{(n-1)/2}\|w\|_{-\Delta_h}\leq\|w\|_{A_h}\ \ \forall w\in H_h.
\label{lower bound AN}
\end{equation}
\end{theorem}
\begin{proof}
Recall that the approximation errors of the equations for all the schemes are defined as
\begin{gather*}
 \psi:=B_h(\rho\Lambda_tu)+\onetwelve h_t^2A_h\Lambda_tu+A_hu-f_N\ \ \text{on}\ \ \omega_{\*h},
\\[1mm]
 \psi^0:=B_h(\rho\delta_tu^0)+\onetwelve h_t^2A_h\Lambda_t\delta_tu^0+\half h_tA_hu_0-u_{1N}-\half h_tf_N^0\ \ \text{on}\ \ \omega_h,
\end{gather*}
cf. formulas \eqref{approxerrN}-\eqref{psie0} for scheme \eqref{num0eq}-\eqref{num0ic}.
For all the
schemes, it was checked above that
\begin{gather}
 \max_{\omega_{\*h}}|\psi|+\max_{\omega_h}|\psi^0|=\mathcal{O}(|\*h|^4).
\label{appr error bound}
\end{gather}
Due to equations for $v$ as well as the definitions of $\psi$ and $\psi^0$,
the error $r:=u-v$ satisfies the following equations
\begin{gather*}
 B_h(\rho\Lambda_tr)+\onetwelve h_t^2A_h\Lambda_tr+A_hr=\psi\ \ \text{on}\ \ \omega_{\*h},
\\[1mm]
 r|_{\partial\omega_{\*h}}=0,\ \
 B_h(\rho\delta_tr^0)+\onetwelve h_t^2A_h\Lambda_t\delta_tr^0+\half h_tA_hr_0=\psi^0\ \ \text{on}\ \ \omega_h,
\end{gather*}
with the approximation errors on the right,
and $r^0=0$.
The stability bound \eqref{energy est1N}, Remark \ref{Bhnu} and estimate \eqref{appr error bound} imply the error bound
\begin{gather*}
 \max_{1\leq m\leq M}\big(\ve_0\|\sqrt{\rho}\bar{\delta}_tr^m\|_h
 +\nu^{-1/2}\|\bar{s}_tr^m\|_{A_h}\big)
 \leq \frac{1}{\ve_0\sqrt{\nu_0\urho}}\big(\|\psi^0\|_h+2I_{h_t}^{M-1}\|\psi\|_h\big)=\mathcal{O}(|\*h|^4).
\end{gather*}

\par Inequality \eqref{lower bound AN} follows from the simple operator inequalities
\[
 a_{\min}^2\tfrac{2}{3}(-\Delta_h)\leq A_N,\ \
 \ve_2a_{\min}^2(\tfrac{2}{3})^2(-\Delta_h)\leq\ve_2\bar{A}_N\leq A_{N\theta},\ \
 a_{\min}^2(\tfrac{2}{3})^{n-1}(-\Delta_h)\leq \bar{A}_N
\]
in $H_h$ respectively for the operators in the first, second and third schemes in Theorem \ref{theo:2}.
\end{proof}

\par Inequality \eqref{lower bound AN} shows that the error norm in Theorem \ref{theo:3} is stronger than the standard mesh energy norm not related to the specific operators in the schemes.

\par Usually $h_t=\mathcal{O}(|h|)$ according to conditions \eqref{stabcondN} and \eqref{condhth2bar}, then $\mathcal{O}(|\*h|^4)=\mathcal{O}(|h|^4)$.

\par Clearly under the hypotheses of Theorem \ref{theo:2}, for example, for scheme \eqref{num1eq}-\eqref{sNbeta} for $n=2$, the general energy conservation law \eqref{en cons law} takes the non-trivial form
\begin{gather*}
 \|\sqrt{\rho}\bar{\delta}_tv^m\|_h^2-\tfrac16 h_t^2\|\bar{\delta}_tv^m\|_{s_{N\beta}^{-1}A_N}^2
+\|\bar{s}_tv^m\|_{s_{N\beta}^{-1}A_N}^2
 =\big(s_{N\beta}^{-1}A_Nv^0,s_tv^0\big)_h
\nonumber\\[1mm]
 +\big(s_{N\beta}^{-1}u_{1N},\delta_tv^0\big)_h
 +\half h_t\big(s_{N\beta}^{-1}f_N^0,\delta_tv^0\big)_h+2I_{h_t}^{m-1}(s_{N\beta}^{-1}f,\mathring{\delta}_tv)_h,\ 1\leq m\leq M.
\label{en cons law 2}
\end{gather*}
The energy conservation laws for the second and third schemes in Theorem \ref{theo:2} are similar.

\section{\large An unconditionally stable finite-difference scheme of the 4th order of approximation}
\label{uncond stable scheme}
\setcounter{equation}{0}
\setcounter{lemma}{0}
\setcounter{theorem}{0}
\setcounter{remark}{0}

Now we discuss the two-level method from \cite[formulas (14), (26)]{HLZ19} constructed for $n=2$.
For $g=0$, in our notation it can be rewritten as a system of two operator equations
\begin{gather}
 \bar{\delta}_tv=c^2\big[I-\onetwelve h_t^2L_h(c^2I)\big]\bar{s}_tw+d\ \ \text{in}\ \ H_h,
\label{two level meth1}
 \\[1mm]
 \bar{\delta}_tw=\big[L_h-\onetwelve h_t^2L_h(c^2L_h)\big]\bar{s}_tv+\tilde{f}\ \ \text{in}\ \ H_h
\label{two level meth2}
\end{gather}
on $\overline\omega_{h_t}\backslash \{0\}$,
where the additional sought function $w$ approximates $\frac{1}{c^2}\partial_tu$ and originally
\[
 L_h:=s_{1N}^{-1}\Lambda_1+s_{2N}^{-1}\Lambda_2=-\bar{s}_N^{\,-1}A_N\ \ \text{for}\ \ n=2.
\]
The given free terms $d$ and $\tilde{f}$ on the right in \eqref{two level meth1}-\eqref{two level meth2} are zero in \cite{HLZ19}, and we have inserted them to cover the case of the non-homogeneous wave equation and for more detailed stability analysis (in practice, $d$ and $\tilde{f}$ are never zero due to the round-off errors).
It is well-known that such type methods are closely related to more standard three-level methods like
\eqref{3level sch1}-\eqref{3level sch2} with $\sigma=\frac14$, for example, see \cite[Section 8]{Z94}.

\par To demonstrate that, we exclude $w$ from this system.
Applying the operators $\frac{1}{c^2}\delta_t$ to \eqref{two level meth1} and
$s_t$ to \eqref{two level meth2}, we find respectively
\begin{gather*}
 \rho\Lambda_tv=\big[I-\onetwelve h_t^2L_h(c^2I)\big]\delta_t\bar{s}_tw+\rho\delta_td,
 \\[1mm]
 s_t\bar{\delta}_tw=L_h\big(I-\onetwelve h_t^2c^2L_h\big)s_t\bar{s}_tv+s_t\tilde{f}.
\end{gather*}
Inserting $s_t\bar{\delta}_tw$ from the second equation into the first one and using the formulas
\[
 \delta_t\bar{s}_tw=s_t\bar{\delta}_tw=\tfrac{\hat{w}-\check{w}}{2h_t},\ \ s_t\bar{s}_tv=v^{(1/4)}\equiv\tfrac14(\hat{v}+2v+\check{v}),
\]
we obtain the following closed equation for $v$
\begin{gather}
 \rho\Lambda_tv+A_h v^{(1/4)}=f_h\ \  \text{on}\ \ \omega_{h_t},
\label{former two level1}
\end{gather}
where we have set
\begin{gather}
 A_h:=\big[I-\onetwelve h_t^2L_h(c^2I)\big](-L_h)\big(I-\onetwelve h_t^2c^2L_h\big),\ \
 f_h:=\big[I-\onetwelve h_t^2L_h(c^2I)\big]s_t\tilde{f}+\rho\delta_td.
\label{former two level2}
\end{gather}
Notice that $A_h^*=A_h>0$ in $H_h$ since $(-L_h)^*=-L_h>0$ and
\begin{gather}
 (A_hy,y)_h=(-L_hz,z)_h,\ \ \text{with}\ \ z:=\big(I-\onetwelve h_t^2c^2L_h\big)y,\ \ \forall y\in H_h.
\label{norms Ah and Lh}
\end{gather}

\par Next, we use the formula $\bar{s}_tw=\check{w}+\half h_t\bar{\delta}_tw$ in equation \eqref{two level meth1} and divide it by $c^2$.
We also use the same formula for $v$ in \eqref{two level meth2} and apply the operator $\half h_t\big[I-\onetwelve h_t^2L_h(c^2I)\big]$ to it:
\begin{gather*}
 \rho\bar{\delta}_tv=\big[I-\onetwelve h_t^2L_h(c^2I)\big](\check{w}+\half h_t\bar{\delta}_tw)+\rho d
\\[1mm]
 = \big[I-\onetwelve h_t^2L_h(c^2I)\big]\check{w}
 +\half h_t\big[I-\onetwelve h_t^2L_h(c^2I)\big]
 \big\{\big[L_h-\onetwelve h_t^2L_h(c^2L_h)\big](\check{v}+\half h_t\bar{\delta}_tv)+\tilde{f}\big\}+\rho d.
\end{gather*}
Considering the first time level $t_1=h_t$, we find
\begin{gather}
 \big(\rho I+\tfrac14h_t^2 A_h\big)\delta_tv^0+\half h_tA_hv^0=u_{1h}+\rho d^1+\half h_tf_h^0,
\label{former two level3}
\end{gather}
where we have set
\begin{gather}
 u_{1h}:=\big[I-\onetwelve h_t^2L_h(c^2I)\big]w^0,\ \ f_h^0:=\big[I-\onetwelve h_t^2L_h(c^2I)\big]\tilde{f}^1+\rho\delta_t d^0,
\label{former two level4}
\end{gather}
with $d^0:=-d^1$ (thus  $\half h_t\rho\delta_t d^0=\rho d^1$), and it is natural to take $w^0=\rho u_1$
on $\bar{\omega}_h$.
\par Since $v^{(1/4)}=v+\frac14 h_t^2\Lambda_tv$, equations \eqref{former two level1} and \eqref{former two level3} form the particular case of method \eqref{3level sch1}-\eqref{3level sch2} for $B_h=I$ and $\sigma=\frac14$,
with $A_h$, $f_h$, $u_{1h}$ and $f_h^0$ given in \eqref{former two level2} and \eqref{former two level4}.

\par We emphasize that the derived three-level method \eqref{former two level1} and \eqref{former two level3} is straightforwardly generalized to any $n\geq 1$ by taking
\begin{gather}
 L_h=s_{1N}^{-1}\Lambda_1+\ldots+s_{nN}^{-1}\Lambda_n =-\bar{s}_N^{\,-1}\bar{A}_N,
\label{oper L any n}
\end{gather}
see formulas \eqref{bsN}-\eqref{bAN}.
Clearly its two-level operator form are the same equations \eqref{two level meth1}-\eqref{two level meth2} with this generalized $L_h$.
\par Let us derive the \textit{unconditional stability} of the generalized method for any $n\geq 1$.
\begin{theorem}
\label{theo:3level version of the 2 level}
For the solution to method \eqref{former two level1}-\eqref{former two level3} and \eqref{oper L any n} for any $n\geq 1$,
the following stability bounds hold:
\begin{gather*}
\max_{1\leq m\leq M}
\big(\|\sqrt{\rho}\bar{\delta}_tv^m\|_h^2+\|\bar{s}_tv^m\|_{A_h}^2\big)^{1/2}
\leq\Big(\|v^0\|_{A_h}^2+\ve_0^{-2}\big\|\tfrac{1}{\sqrt{\rho}}u_{1h}\big\|_h^2\Big)^{1/2}
       +2\big\|\tfrac{1}{\sqrt{\rho}}f_h\big\|_{L_{h_t}^1(H_h)},
\label{energy est1 2level}
\end{gather*}
for any free terms $f_h$: $\{t_m\}_{m=0}^{M-1}\to H_h$ and $u_{1h}\in H_h$,
where the $f_h$-term can be replaced with
$2I_{h_t}^{M-1}\|A_h^{-1/2}\bar{\delta}_tf_h\|_h+3\max\limits_{0\leq m\leq M-1}\|A_h^{-1/2}f_h^m\|_h$;
\begin{gather}
\max_{0\leq m\leq M}
\max\big\{\|\sqrt{\rho}v^m\|_h,\,\|I_{h_t}^m\bar{s}_tv\|_{A_h}\big\}
 \leq\|\sqrt{\rho}v^0\|_h
\nonumber\\[1mm]
 +2\|(-L_h)^{-1/2}w^0\|_h
 +\tfrac{h_t}{2}\|(-L_h)^{-1/2}\tilde{f}^1\|_h
 +2 I_{h_t}^{M-1}\|(-L_h)^{-1/2}s_t\tilde{f}\|_h
 +2I_{h_t}^{M}\|{\sqrt{\rho}}g\|_h,
\label{energy est2 2level}
\end{gather}
for any $d$, $\tilde{f}$: $\{t_m\}_{m=1}^{M}\to H_h$ and $w^0\in H_h$,
together with the energy conservation law
\begin{gather*}
 \|\sqrt{\rho}\bar{\delta}_tv^m\|_h^2+\|\bar{s}_tv^m\|_{A_h}^2
 =\big(A_hv^0,s_tv^0\big)_h
 +\big(u_{1h}+\half h_t f_h^0,\delta_tv^0\big)_h
 +2I_{h_t}^{m-1}(f_h,\mathring{\delta}_tv)_h,\ 1\leq m\leq M.
\end{gather*}
\end{theorem}
\begin{proof}
The first stability bound, the second stability bound in the form
\begin{gather}
\max_{0\leq m\leq M}
\max\big\{\|\sqrt{\rho}v^m\|_h,\,\|I_{h_t}^m\bar{s}_tv\|_{A_h}\big\}
 \leq\|\sqrt{\rho}v^0\|_h
 +2\|A_h^{-1/2}u_{1h}\|_h
 +2\|A_h^{-1/2}f_h\|_{L_{h_t}^1(H_h)}
\label{energy est2 2level 0}
\end{gather}
and the stated energy conservation law directly {follow}
from general Theorem \ref{theo:1} and law \eqref{en cons law}
in the case $B_h=I$ and $\sigma=\tfrac14$ (recall that then $\ve_0=1$).
In addition, the term $\rho\delta_td$ can be extracted from $f_h$ in \eqref{energy est2 2level 0} and added as $2I_{h_t}^{M}\|{\sqrt{\rho}}g\|_h$ (since $s_tg^0=0$) on the right like it stands in \eqref{energy est2 2level}.
Notice that the bounds and the law are especially simplified in this particular case.
\par Moreover, the following chain of transformations hold
\begin{gather*}
 \|A_h^{-1/2}w\|_h^2=(A_h^{-1}w,w)_h
 =\big(\big(I-\onetwelve h_t^2c^2L_h\big)^{-1}(-L_h)^{-1}\big[I-\onetwelve h_t^2L_h(c^2I)\big]^{-1}w,w\big)_h
\\[1mm]
 =\|(-L_h)^{-1/2}\big[I-\onetwelve h_t^2L_h(c^2I)\big]^{-1}w\|_h^2
 \ \ \forall w\in H_h,
\end{gather*}
cf. \eqref{norms Ah and Lh}.
This result allows us to pass from the norms of $u_{1h}$ and $f_h$ given in \eqref{energy est2 2level 0} to norms of $w^0$ and $\tilde{f}$ standing in bound \eqref{energy est2 2level}.
\end{proof}
\par Note that here the norms $\|\cdot\|_{A_h}$ can be rewritten in terms of $\|\cdot\|_{-L_h}$ and $L_h$ according to formula
\eqref{norms Ah and Lh} that remains valid for any $n\geq 1$.
\par We finally emphasize that clearly the operator $A_h$ and the right-hand terms $f_h$ and $u_{1h}$, see \eqref{former two level2} and
\eqref{former two level4},
with $L_h$ given in \eqref{oper L any n}, and consequently the implementation of the method are much more complicated than the corresponding operators and the right-hand terms in the schemes constructed in Section \ref{numerovschemes} since the latter ones do not contain neither non-explicit (inverse) operators nor powers of the mesh operators.

\section{\large The case of non-uniform meshes in space and time}
\label{non inif mesh}
\setcounter{equation}{0}
\setcounter{lemma}{0}
\setcounter{theorem}{0}
\setcounter{remark}{0}

In this Section, we briefly dwell on the case of non-uniform rectangular meshes in $x$ and $t$ when the schemes can be extended following \cite{ZK20}.
Note that this is necessary, in particular, for extending the schemes to more general domains including those composed from rectangular parallelepipeds or for implementing a dynamic choice of the time step.
We confine ourselves only by the scheme like \eqref{num3eq}-\eqref{num3ic} for any $n\geq 1$
and emphasize  that the scheme now will be constructed directly, \textit{without} considering intermediate schemes like above in Section \ref{numerovschemes}.
\par Define the general non-uniform meshes $\overline\omega_{h_t}$ in $t$ and $\bar{\omega}_{hk}$ in $x_k$ with the nodes
\[
 0=t_0<t_1<\ldots<t_M=T,\ \ 0=x_{k0}<x_{k1}<\ldots<x_{kN_k}=X_k
\]
and the steps $h_{tm}=t_m-t_{m-1}$ and $h_{kl}=x_{kl}-x_{k(l-1)}$, $1\leq k\leq n$.
Let $\omega_{hk}=\{x_{kl}\}_{l=1}^{N_k-1}$.
We set
\begin{gather*}
 h_{t+,\,m}=h_{t(m+1)},\ \
 h_{*t}=\half(h_t+h_{t+}),\ \ h_{k+,\,l}=h_{k(l+1)},\ \ h_{*k}=\half(h_k+h_{k+})
\end{gather*}
and define also the maximal mesh steps
\begin{gather*}
 h_{t\max}=\max_{1\leq m\leq M}h_{tm},\ \
 h_{\max}=\max_{1\leq k\leq n}\max_{1\leq l\leq N_k} h_{kl},\ \ \*h_{\max}=\max\,\{h_{\max},h_{t\max}\}.
\end{gather*}
Let now
$\bar{\omega}_{h}=\bar{\omega}_{h1}\times\ldots\times\bar{\omega}_{hn}$,
$\omega_{h}=\omega_{h1}\times\ldots\times\omega_{hn}$ and $\partial\omega_h=\bar{\omega}_{h}\backslash\omega_{h}$.

\par We generalize the above defined difference operators in $t$ and $x_k$ as
\begin{gather*}
 \delta_ty=\tfrac{1}{h_{t+}}(\hat{y}-y),\ \
 \bar{\delta}_ty=\tfrac{1}{h_{t}}(y-\check{y}),\ \ \Lambda_ty=\tfrac{1}{h_{*t}}(\delta_ty-\bar{\delta}_ty),\ \
\\[1mm]
 \Lambda_kw_l=\tfrac{1}{h_{*k}}\big[\tfrac{1}{h_{k(l+1)}}(w_{l+1}-w_l)-\tfrac{1}{h_{kl}}(w_l-w_{l-1})\big],
 \ \ \text{with}\ \ w_l=w(x_{kl}).
\end{gather*}
Next we generalize the above averaging technique including the following average in $x_k$:
\begin{gather*}
 q_kw(x_{kl})=\frac{1}{h_{*k,l}}\int_{I_{kl}}w(x_k)e_{kl}(x_k)\,dx_k,
\\[1mm]
 \text{with}\ \
 e_{kl}(x_k)=\tfrac{x_k-x_{k(l-1)}}{h_{kl}}\ \ \text{on}\ \ [x_{k(l-1)},x_{kl}],\
 e_{kl}(x_k)=\tfrac{x_{k(l+1)}-x_k}{h_{k(l+1)}}\ \ \text{on}\ \ [x_{kl},x_{k(l+1)}].
\end{gather*}
For a function $w(x_k)$ smooth on $[0,X_k]$, formula \eqref{qklambdak} remains valid.
Also now we have
\begin{gather}
 q_kw=w+q_k\rho_{k1}(\partial_kw),
\nonumber\\[1mm]
 q_kw=w+\tfrac13(h_{k+}-h_k)\partial_kw+\onetwelve\big[(h_{k+})^2-h_{k+}h_k+h_k^2\big]\partial_k^2w+q_k\rho_{k3}(\partial_k^3w)
\label{qkexpansion}
\end{gather}
on $\omega_{hk}$.
The first bound \eqref{resid_bound} is now valid for $s=1,3$, with $h_k$ replaced with $h_{*k}$,
that follows from Taylor's formula after calculating the arising integrals of polynomials over $I_{kl}$ and using residual \eqref{taylor_residual}.
Next, once again due to Taylor's formula, we derive
\begin{gather}
 \partial_kw=\half(\bar{\delta}_kw+\delta_kw)-\tfrac14(h_{k+}-h_k)\partial_k^2w+\rho_{k}^{(1)}(\partial_k^3w),\ \
 \partial_k^2w=\Lambda_kw+\rho_{k3}^{(2)}(\partial_k^3w),
\label{p1 p2}\\[1mm]
 |\rho_{k}^{(s)}(\partial_k^3w)|\leq c^{(s)}h_{*k}^{3-s}
 \|\partial_k^3w\|_{C(I_{kl})},\ \ s=1,2,
\label{p1 p2 err}
\end{gather}
on $\omega_{hk}$.
Inserting expansions \eqref{p1 p2} into
expansion \eqref{qkexpansion} and using \eqref{p1 p2 err} lead to the formulas
\begin{gather}
 q_kw=s_{kN}w+\tilde{\rho}_{k3}(\partial_k^3w),\ \
 |\tilde{\rho}_{k3}(\partial_k^3w)|\leq \tilde{c}_3h_{*k}^3\|\partial_k^3w\|_{C(I_{kl})},
\label{qk skN}
\end{gather}
on $\omega_{hk}$, with the generalized Numerov-type averaging operator in $x_k$
\begin{gather*}
 s_{kN}:=I+\tfrac13(h_{k+}-h_k)\big[\half(\bar{\delta}_k+\delta_k)-\tfrac14(h_{k+}-h_k)\Lambda_k\big]
 +\onetwelve\big[(h_{k+})^2-h_{k+}h_k+h_k^2\big]\Lambda_k
\nonumber\\[1mm]
 =I+\tfrac16(h_{k+}-h_k)(\bar{\delta}_k+\delta_k)+\onetwelve h_kh_{k+}\Lambda_k.
\end{gather*}
Consequently the following two more forms for $s_{kN}$ also hold
\begin{gather*}
 s_{kN}w_l
 =w_l+\onetwelve[(h_{k+}\beta_{k}\delta_k-h_{k}\alpha_{k}\bar{\delta}_k)w]_l
 =\onetwelve(\alpha_{kl}w_{l-1}+10\gamma_{kl}w_l+\beta_{kl}w_{l+1}),
\\[1mm]
 \text{with}\ \
 \alpha_k=2-\tfrac{h_{k+}^2}{h_kh_{*k}},\ \beta_k=2-\tfrac{h_k^2}{h_{k+}h_{*k}},\ \gamma_k=1+\tfrac{(h_{k+}-h_k)^2}{5h_k h_{k+}},\ \alpha_k+10\gamma_k+\beta_k=12
\end{gather*}
on $\omega_{hk}$.
Note that other derivations and forms for $s_{kN}$ can be found in \cite{JIS84,ChS18,RCM14}.

\par Quite similarly the following formulas with the generalized average $q_tw=q_{n+1}w$ and the Numerov-type operator $s_{tN}$ in $t$ hold on $\omega_{h_t}$:
\begin{gather}
 q_tw=s_{tN}w+\tilde{\rho}_{t3}(\partial_t^3w),\ \
|\tilde{\rho}_{t3}(\partial_t^3w)|\leq \tilde{c}_3h_{*t}^3\|\partial_t^3w\|_{C[t_{m-1},t_{m+1}]},
\label{qt stN}\\[1mm]
 s_{tN}y=y+\onetwelve(h_{t+}\beta_t\delta_t-h_t\alpha_t\bar{\delta}_t)y
 =\onetwelve(\alpha_t\check{y}+10\gamma_ty+\beta_t\hat{y}),
\nonumber\\[1mm]
 \text{with}\ \ \alpha_t=2-\tfrac{h_{t+}^2}{h_th_{*t}},\ \beta_t=2-\tfrac{h_t^2}{h_{t+}h_{*t}},\,
 \gamma_t=1+\tfrac{(h_{t+}-h_t)^2}{5h_t h_{t+}}.
\nonumber
\end{gather}

\par Let the operators $\bar{s}_N$, $\bar{s}_{N\widehat{l}}$ and $\bar{A}_N$ be defined as in \eqref{bsN}-\eqref{bAN} but with the generalized terms $s_{kN}$ and $\Lambda_k$.
Formula \eqref{avereq} for $u$ remains valid and due to expansions \eqref{qk skN}-\eqref{qt stN} implies
\[
 \bar{s}_N(\rho\Lambda_tu)
 -(a_1^2\bar{s}_{N\widehat{1}}\Lambda_1+\ldots+a_n^2\bar{s}_{N\widehat{n}}\Lambda_n)s_{tN}u
 =\bar{q}q_tf+O(\*h_{\max}^3)\ \ \text{on}\ \ \omega_{\*h}.
\]

\par Formula \eqref{avereq_0} for $u$ remains valid as well, where $q_ty^0$ is given by formula \eqref{qty0} with $h_{t1}$ instead of $h_t$.
It concerns only time levels $t_0=0$ and $t_1=h_{t1}$ thus easily covers the case of the non-uniform mesh in $t$ and implies now
\[
 \bar{s}_N(\rho\delta_tu)^0
 =\bar{q}(\rho u_1)+(a_1^2\bar{s}_{N\widehat{1}}\Lambda_1+\ldots+a_n^2\bar{s}_{N\widehat{n}}\Lambda_n)
 \big[\tfrac{h_{t1}}{2}u_0+\tfrac{h_{t1}^2}{12}u_1+\tfrac{h_{t1}^2}{12}(\delta_tu)^0\big]+\bar{q}q_tf^0+O(\*h_{\max}^3)
\]
on $\omega_h$, cf. \eqref{aver_equat 0}.

\par Due to the above formulas for $\Lambda_t$ and $s_{tN}$ as well as expansions \eqref{qk skN}-\eqref{qt stN}, the last two expansions for $u$ with omitted $O(\*h_{\max}^3)$-terms imply the generalized scheme \eqref{num3eq}-\eqref{num3ic} on the non-uniform mesh
\begin{gather}
 \tfrac{1}{h_{*t}}\big\{\bar{s}_N(\rho\delta_tv)+\tfrac{h_{*t}h_{t+}}{12}\beta_t\bar{A}_N\delta_tv
                  -\big[\bar{s}_N(\rho\bar{\delta}_tv)+\tfrac{h_{*t}h_t}{12}\alpha_t\bar{A}_N\bar{\delta}_tv\big]\big\}
 +\bar{A}_Nv=\bar{s}_Ns_{tN}f\ \ \text{on}\ \ \omega_{\*h},
\label{num3eq nonuni}\\[1mm]
 v|_{\partial\omega_{\*h}}=g,\ \bar{s}_N(\rho\delta_tv)^0+\tfrac{h_{t1}^2}{12}\bar{A}_N(\delta_tv)^0
 +\tfrac{h_{t1}}{2}\bar{A}_Nv_0
 =\bar{s}_N(\rho u_1)-\tfrac{h_{t1}^2}{12}\bar{A}_Nu_1+\tfrac{h_{t1}}{2}f_N^0\ \ \text{on}\ \ \omega_h,
\label{num3ic nonuni}
\end{gather}
with $f_N^0=\bar{s}_Nf_0+\tfrac{h_{t1}}{3}(\delta_tf)^0$.
Its equations
have the approximation errors of the order $O(\*h_{\max}^3)$.
\par For the uniform mesh in $t$, the left-hand side of \eqref{num3eq nonuni} takes the form like above in \eqref{num3eq}:
\[
 \bar{s}_N(\rho\Lambda_tv)+\onetwelve h_t^2\bar{A}_N\Lambda_tv+\bar{A}_Nv=\bar{s}_Ns_{tN}f,
\]
and the equation has the higher approximation order $O(h_{\max}^3+h_{t\max}^4)$
due to relations \eqref{qkw 3}-\eqref{resid_bound} for $k=n+1$.

\par Other above constructed schemes can be also generalized to the case of non-uniform meshes in the similar manner.
In addition, one can check also that the approximation errors still has the 4th order $O(\*h_{\max}^4)$ for non-uniform meshes with slowly varying mesh steps, cf. \cite{Z15},
provided that, for example, $f_N^0=\bar{s}_Nf^0-f^0+f_{dh_t}^{(0)}$.

\par Here we do not intend to study the stability issue in the case of the non-uniform mesh (even only in space) which
is essentially more cumbersome since the operators $s_{kN}$ are not self-adjoint as well as $s_{kN}$ and $\Lambda_k$ do not commute any more.
Moreover, this can lead to much stronger conditions on $h_t$, especially in the case when the corresponding eigenvalue problem \eqref{geneigval} has complex eigenvalues, see \cite{Z15,ZC20}.
On the other hand, for smoothly varying mesh steps and not only, results of 1D numerical experiments are positive, see \cite{Z15,ZK20}.

\section{\large Iterative methods and numerical experiments}
\label{numerexperiments}
\setcounter{equation}{0}
\setcounter{lemma}{0}
\setcounter{theorem}{0}
\setcounter{remark}{0}

\textbf{\ref{numerexperiments}.1.} We go back to equation \eqref{laeq} at the upper time level, or omitting the superscript $m$ and taking $\sigma=\onetwelve$, to the equation
\begin{gather}
 B_h(\rho w)+\tfrac{1}{12}h_t^2A_hw=b\ \ \text{in}\ \ H_h,
\label{eqBAb}
\end{gather}
with any commuting operators $B_h^*=B_h>0$ and $A_h^*=A_h>0$, in particular, for all pairs of operators $(B_h,A_h)$ considered in Section \ref{numerovschemes}.
Thus we assume that the non-homogeneous boundary condition $v|_{\partial\omega_{\*h}}=g$ is reduced to the homogeneous one $v|_{\partial\omega_{\*h}}=0$ by respective change in $f_N$ and $u_{1N}$ at the mesh nodes of $\omega_h$ closest to $\partial\omega_h$.

\par We first consider the one-step iterative method with a constant parameter $\theta>0$:
\begin{gather}
 B_h\big(\rho\tfrac{w^{(l+1)}-w^{(l)}}{\theta}\big)+B_h(\rho w^{(l)})+\tfrac{1}{12}h_t^2A_hw^{(l)}=b,\ \ l\geq 0,
\label{eqBAb2}
\end{gather}
where $B_h$ serves as a preconditioner.
Its equivalent practical form is
\begin{gather}
 w^{(l+1)}=w^{(l+1)}(\theta):=(1-\theta)w^{(l)}-\tfrac{\theta}{\rho}B_h^{-1}\big(\tfrac{1}{12}h_t^2A_hw^{(l)}-b\big),\ \ l\geq 0.
\label{eqBAb2 new}
\end{gather}
For schemes from Section \ref{numerovschemes},
{application of}  $B_h^{-1}$ can be effectively implemented by FFT.
\begin{theorem}
\label{theo: 1st iter meth}
Let the stability condition \eqref{stabcondN} on $h_t$ be valid for some $0<\ve_0<1$.

\par For the one-step iterative method \eqref{eqBAb2} with the parameter
$\theta:=\theta_{opt}=2/(1+\bar{\lambda}(\ve_0^2))$, where $\bar{\lambda}(\ve_0^2):=1+\half(1-\ve_0^2)$,
the convergence rate estimate holds
\begin{gather}
 \|w-w^{(l)}\|\leq q_0^l\|w-w^{(0)}\|,\ \ l\geq 0,\ \ \forall w^{(0)}\in H_h,
\label{conv rate 1 iter meth 1}
\end{gather}
in two norms $\|\cdot\|=\|\sqrt{\rho}\cdot\|_h$ and $\|\cdot\|_{\mathcal{A}_{\*h}}$, with $\mathcal{A}_{\*h}:=D_\rho+\tfrac{1}{12}h_t^2B_h^{-1}A_h$ and
\[
 q_0=q_0(\ve_0^2):=\frac{\bar{\lambda}(\ve_0^2)-1}{\bar{\lambda}(\ve_0^2)+1}=\frac{1-\ve_0^2}{5-\ve_0^2}\leq 0.2\ \ \text{on}\ \ [0,1).
\]
\end{theorem}
\begin{proof}
We rewrite equation \eqref{eqBAb} and the iterative method \eqref{eqBAb2} in the canonical forms
\begin{gather}
 \mathcal{A}_{\*h}w=\tilde{b}:= B_h^{-1}b,\ \
 D_\rho w^{(l+1)}=D_\rho w^{(l)}-\theta(\mathcal{A}_{\*h}w^{(l)}-\tilde{b}),\ \ l\geq 0,
\label{IDpBAI}
\end{gather}
with the preconditioner $D_\rho$.
Recall that $D_\rho^*=D_\rho>0$ and $\mathcal{A}_{\*h}^*=\mathcal{A}_{\*h}>0$.
Moreover, under condition \eqref{stabcondN}, the following spectral equivalence inequalities hold
\begin{gather}
  D_\rho \leq \mathcal{A}_{\*h}=D_\rho+\tfrac{1}{12}h_t^2B_h^{-1}A_h
  \leq \bar{\lambda}(\ve_0^2)D_\rho\ \ \text{in}\ \ H_h,
\ \ \text{with}\ \ \bar{\lambda}(\ve_0^2)=1+\half(1-\ve_0^2).
\label{spectr equiv}
\end{gather}
Thus according to the theory of iterative methods in the form \eqref{IDpBAI}, for example, see \cite{SN89}, the optimal vaue of the parameter $\theta$ is $\theta_{opt}$, and the convergence rate estimate \eqref{conv rate 1 iter meth 1} is valid.
\end{proof}

\par We also can consider the $N$-step iterative method with the Chebyshev parameters
\begin{gather}
 w^{(l+1)}=(1-\theta^{(l)})w^{(l)}-\tfrac{\theta^{(l)}}{\rho}B_h^{-1}\big(\tfrac{1}{12}h_t^2A_hw^{(l)}-b\big),\ \
\label{eqBAb2 Cheb 1}\\[1mm]
 \theta^{(l)}:=\frac{\theta_{opt}}{1+q_0\cos\frac{\pi(l+1/2)}{N}},\ \ l=0,\ldots,N-1,
\label{eqBAb2 Cheb 2}
\end{gather}
see much more details in \cite{SN89}.
\begin{theorem}
\label{theo: 2nd iter meth}
Let condition \eqref{stabcondN} on $h_t$ be valid for some $0<\ve_0<1$.
For the $N$-step iterative method \eqref{eqBAb2 Cheb 1}-\eqref{eqBAb2 Cheb 2}, the convergence rate estimate holds
\begin{gather*}
 \|w-w^{(N)}\|\leq \tfrac{2q_1^N}{1+2q_1^N}\|w-w^{(0)}\|\ \ \forall w^{(0)}\in H_h,
\label{conv rate 4 iter meth 1}
\end{gather*}
in two norms $\|\cdot\|=\|\sqrt{\rho}\cdot\|_h$ and $\|\cdot\|_{\mathcal{A}_{\*h}}$, with
\[
 q_1=q_1(\ve_0^2)
 :=\frac{{\bar{\lambda}^{1/2}(\ve_0^2)}-1}{{\bar{\lambda}^{1/2}(\ve_0^2)}+1}
 =\frac{1-\ve_0^2}{5-\ve_0^2+4\sqrt{1+\half(1-\ve_0^2)}}
 \leq\frac{1}{5+4\sqrt{1.5}}\approx 0.1010\ \ \text{on}\ \ [0,1).
\]
\end{theorem}
\begin{proof}
The result is valid due to the theory of the $N$-step iterative methods, for example, see \cite{SN89}, taking into account the spectral equivalence inequalities \eqref{spectr equiv}.
\end{proof}

\par Let us discuss the convergence rates of the suggested iterative methods.
Importantly, $q_0$ and $q_1$ are \textit{independent of both the meshes and $\rho$}, in particular, the spread of its values $\hat{\rho}=\orho/\urho$ with $\rho(x)\leq \orho$ on $\bar{\Omega}$.
The last point is essential for some applications.
In the typical case $\ve_0^2=\half$, one has $q_0(\half)=\frac19\approx 0.1111$.
For the $\sqrt2$ times stronger condition \eqref{stabcondN} on $h_t$ with
$\ve_0^2=\frac34$, one has already $q_0(\frac34)\approx 0.05882$.
Recall that often the much higher common ratio $q_0=0.5$ is considered as good.
\par One has also, in particular, $q_1(\half)\approx 0.05573$ and $q_1(\frac34)\approx 0.02944$.
It is easy to see that
\[
 0.5<\frac{q_1(\ve_0^2)}{q_0(\ve_0^2)}\leq\frac{5}{5+4\sqrt{1.5}} \approx 0.5051\ \ \text{on}\ \ [0,1),
\]
thus the iterative method \eqref{eqBAb2 Cheb 1}-\eqref{eqBAb2 Cheb 2} is much faster than \eqref{eqBAb2},
as well as $q_0,q_1$ and $\frac{q_1}{q_0}$ decrease on $[0,1)$.
Moreover, $q_l(\ve_0^2)\to 0$ as $\ve_0\to 1-0$, $l=0,1$, i.e., the common ratios become \textit{arbitrarily small} as condition \eqref{stabcondN} on $h_t$ turns more and more stronger.

\par It is well-known that often the variational counterparts of the above iterative methods, namely, the steepest descent and conjugate gradient methods are more preferable.
Here we do not come into details and mention only that in the former method the parameter $\theta=\theta_l$ is defined such that
\begin{gather*}
 \|w-w^{(l+1)}(\theta_l)\|_{\mathcal{A}_{\*h}}=\min_{\theta>0}\|w-w^{(l+1)}(\theta)\|_{\mathcal{A}_{\*h}}.
\label{steep desc}
\end{gather*}
The explicit formula for $\theta_l$ (for example, see \cite{SN89}) is given by the formula
\[
 \theta_l=\frac{(D_\rho y^{(l)},y^{(l)})_h}{(\mathcal{A}_{\*h}y^{(l)},y^{(l)})_h}
 =\frac{\|\sqrt{\rho}y^{(l)}\|_h^2}{\|\sqrt{\rho}y^{(l)}\|_h^2+\tfrac{1}{12}h_t^2\big(B_h^{-1}A_hy^{(l)},y^{(l)}\big)_h},\,\
 y^{(l)}:=w^{(l)}+\tfrac{1}{\rho}B_h^{-1}\big(\tfrac{1}{12}h_t^2A_hw^{(l)}-b\big).
\]

\par The above iterative methods can be generalized for equation \eqref{eqBAb} with any $\sigma\neq 0$ instead of $\frac{1}{12}$ that is essential, in particular, for implementation of the scheme from Section \ref{uncond stable scheme} (no methods to this end were described in \cite{HLZ19}).

\par Concerning the initial guess for methods \eqref{eqBAb2} and \eqref{eqBAb2 Cheb 1}-\eqref{eqBAb2 Cheb 2}, one can base simply on the formula $v^{m+1,(0)}=v^m$, for $0\leq m\leq M-1$, or $v^{m+1,(0)}=2v^m-v^{m-1}$, for $1\leq m\leq M-1$.
But it seems much better to use closely related equations \eqref{3level sch1}-\eqref{3level sch2} for $\sigma=0$ in the form:
\begin{gather}
 (\Lambda_tv)^{m,(0)}=-\tfrac{1}{\rho}B_h^{-1}(A_hv^m-f^m)\ \ \text{in}\ \ H_h,\ \ 1\leq m\leq M-1,
\label{initial1}\\[1mm]
 (\delta_tv^0)^{(0)}=-\tfrac{1}{\rho}B_h^{-1}\big(\half h_tA_hv^0-u_1-\half h_tf^0\big)\ \ \text{in}\ \ H_h,
\nonumber
\end{gather}
and this expectation is confirmed in numerical experiments.
Here applying $B_h^{-1}$ can be again effectively implemented by FFT.
Note that a discussion on the choice of the initial guess can be found in \cite{BTT18}.

\smallskip\textbf{\ref{numerexperiments}.2.}
Now we describe results of our numerical experiments.
To be definite, we take $n=2$ and use mainly scheme \eqref{num1eq}-\eqref{num1ic} that below we call \textit{scheme $S_0$};
we also apply the second formula \eqref{ftd02} to compute $f_N^0$.
In order to compare the results with those presented
in literature, we solve two test problems from \cite{HLZ19} including the wave propagation in a
the three-layer medium for the square spatial mesh and also take one more problem for the rectangular one.
Our numerical tests have been performed on the computer with
Intel\textsuperscript{\textregistered} Xeon\textsuperscript{\textregistered}
processor E5-2670, 8GB RAM, and the algorithm has been implemented using $C$++ language.
\par We rewrite the IBVP \eqref{hyperb2eq}-\eqref{hyperb2ibc} for $n=2$ and $g=0$ as
\begin{gather*}
 \partial_t^2 u - c^2(x,y)(\partial_x^2u+\partial_y^2u)=\varphi(x,y,t)\ \ \text{for}\ \ (x,y) \in [0,X]\times [0,X],\ 0 <t\leq T,
\\[1mm]
 u|_{\Gamma_T}=0,\ \
 u(x,y,0) = u_0(x,y),\ \ \partial_tu(x,y,0)=u_1(x,y)\ \ \text{for}\ \ (x,y) \in [0,X]\times [0,X].
\end{gather*}

{\bf Example 1.} First we take $X=T=2$, $c^2(x,y) = 1+\big(\frac{\pi x}{8}\big)^2+\big(\frac{\pi y}{8}\big)^2$.
The data $u_0(x)$, $u_1=0$ and $\varphi(x,y,t)$ are chosen so that the solution is the simple standing wave
$u(x,y,t) =\sin(\pi x)\sin(\pi y)\cos(\pi t)$ as in \cite{HLZ19}.
\par Table~\ref{tab1} contains the errors $e_{L^2}(N)$ and $e_{L^\infty}(N)$ in the mesh $L_2$ and $L_\infty$ norms (i.e., in $H_h$ and the mesh uniform norms) at $t=T$
together with the corresponding experimental convergence rates:
\[
 p_{L^q}(N) = \Big.\log \frac{e_{L^q}(N)}{e_{L^q}(N/2)}\Big\slash \log 2,\ \
 q=2,\infty.
\]
Here we take $h_x=h_y=h=\frac{X}{N}$.
Also hereafter
$N_{iter}$ denotes the maximal number of iterations \eqref{eqBAb2}
required to solve the systems of equations with the given tolerance $10^{-10}$.
CPU time is also included.
Several spatial steps $h$
are used, and due to the stability condition the time step is restricted to $h_t=0.25h$.
\begin{table}[ht]
\begin{center}
\caption{Example 1: errors $e_{L^q}(N)$,
convergence rates $p_{L^q}(N)$, numbers of iterations $N_{iter}$
and CPU times for a sequence of meshes}
\label{tab1}
\vspace{2mm}
\begin{tabular}{lccccccc}
\hline\noalign{\smallskip}
$N$ & $h_x= h_y$ & $e_{L^2}(N)$ & $ p_{L_2}(N)$ & $e_{L^\infty}(N)$ & $ p_{L_\infty}(N)$ & $N_{iter}$ & CPU time  \\
\noalign{\smallskip}
\hline
\noalign{\smallskip}
 8 & 1/4  &  3.3660e-3&  ---   & 3.5483e-3 & ---  & 6 & $0.001$ s  \\
16 & 1/8  & 2.0104e-4 &  4.065 & 2.2719e-4 & 3.965 & 6 & $0.012$ s \\
32 & 1/16 & 1.2128e-5 &  4.051 & 1.4623e-5 & 3.958 & 5 & $0.085$ s \\
64 & 1/32 & 7.4564e-7 &  4.023 & 9.1493e-7 & 3.998 & 5 & $0.608$ s \\
\hline
\end{tabular}
\end{center}
\end{table}

\par Clearly scheme  $S_0$ demonstrates the 4th order accuracy in both norms.
The obtained $L^2$ errors are about 5 times more accurate than those in \cite[Table 12]{HLZ19}.
Also it can be seen that $N_{iter}$ is small, and the CPU time is approximately proportional
to the size of the discrete problem.

\par Next we investigate in more details  the convergence of the proposed iterative method \eqref{eqBAb2}
with the initial guess defined by \eqref{initial1}.
The given problem is solved for different values of $\varepsilon_0^2$ and the number $M$ defining the time step $h_t=\frac{T}{M}$.
Table~\ref{tab1b} contains the values of $N_{iter}$
for $h_x=\frac{1}{32}$.
For comparison, in brackets we also present its values
when a simple guess $w^{(0)}=w$ is used.
We observe that
the convergence of the iterative method \eqref{eqBAb2} with the initial guess defined by \eqref{initial1} is very fast requiring no more than 5 iterations to reach the high tolerance
$10^{-10}$, and its rate is only slightly sensitive to the value of the parameter $\theta$.
The role of this initial guess is essential since it reduces $N_{iter}$ at least twice.
Still this dependence can become more pronounced for not so smooth solutions when errors in high modes are
more important.
\begin{table}[ht]
\begin{center}
\caption{Example 1:
$N_{iter}$ for different $M$ and parameters $\theta$ in \eqref{eqBAb2}.}
\label{tab1b}
\vspace{2mm}
\begin{tabular}{lccc}
\hline
\noalign{\smallskip}
$M$ & $\theta =\frac89\ (\varepsilon_0^2=\frac12)$ & $\theta=\frac{16}{17}\ (\varepsilon_0^2=\frac34)$
& $ \theta = \frac{32}{33}\ (\varepsilon_0^2=\frac78)$\\
\noalign{\smallskip}
\hline
\noalign{\smallskip}
256 & 5 (10) &  5 (9) &  5 (9) \\
512 & 5 (10) &  4 (9) &  4 (8) \\
1024 & 4 (10) &  4 (9) &  3 (8) \\
2048 & 4 (9)$\hphantom0$ &  3 (8) &  3 (8) \\
\hline
\end{tabular}
\end{center}
\end{table}

\smallskip\par{\bf Example 2.}
Next we take $X=T=1$,  $c^2(x,y)=(1+x^2+4y^2)^{-1}$.
The data $u_0$, $u_1$ and $\varphi$ are chosen so that the solution is the simple standing wave
$u(x,y,t) =\sin(\pi x)\sin(4 \pi y)\exp(t)$.
In this example, the wave propagation in $x$ and $y$ directions is different, thus the mesh steps $h_x=\frac{1}{N}\neq h_y=\frac{1}{4N}$ are taken.

\par Table~\ref{tab2a} contains the errors $e_{L^2}(N)$ and $e_{L^\infty}(N)$
at $t=1$ together with the corresponding experimental convergence rates for scheme $S_0$.
Clearly the scheme is robust for $h_x\neq h_y$ as well.
\begin{table}[ht]
\begin{center}
\caption{Example 2: errors $e_{L^q}(N)$ and convergence rates $p_{L^q}(N)$ of
the solution to scheme $S_0$, i.e., \eqref{num1eq}-\eqref{num1ic}, for a sequence of meshes}
\label{tab2a}
\vspace{2mm}
\begin{tabular}{lccccccc}
\hline\noalign{\smallskip}
$N$ & $h_x$ & $ h_y$ & $h_t$ & $e_{L^2}(N)$ & $ p_{L_2}(N)$ & $e_{L^\infty}(N)$ & $ p_{L_\infty}(N)$ \\
\noalign{\smallskip}
\hline
\noalign{\smallskip}
4  & 1/4  & 1/16  & 1/32  &  3.3710e-3 &  ---   & 3.6410e-3 & ---   \\
8  & 1/8  & 1/32  & 1/64  &  1.9822e-4 &  4.088 & 2.3470e-4 & 3.955 \\
16 & 1/16 & 1/64  & 1/128 &  1.1960e-5 &  4.051 & 1.4849e-5 & 3.982 \\
32 & 1/32 & 1/128 & 1/256 &  7.2937e-7 &  4.035 & 9.2547e-7 & 4.004 \\
\hline
\end{tabular}
\end{center}
\end{table}

\par For comparison, we solve the same problem by using the modified 4th order scheme \eqref{num2eq}-\eqref{num2ic} (suitable for any $n$) and put the same type results in Table~\ref{tab2b}.
The results for both schemes are very close thus for other tests we apply only the former one.
Nevertheless we note carefully that all the errors are (very) slightly larger for the latter scheme;
this is since it exploits the more dissipative in space operator $\bar{s}_N=s_N+\frac{h_x^2}{12}\frac{h_y^2}{12}\Lambda_x\Lambda_y$ rather than $s_N$ in the former scheme.
\begin{table}[ht]
\begin{center}
\caption{Example 2: errors $e_{L^q}(N)$ and convergence rates $p_{L^q}(N)$ of
the solution to scheme \eqref{num2eq}-\eqref{num2ic} for a sequence of meshes}
\label{tab2b}
\vspace{2mm}
\begin{tabular}{lccccccc}
\hline\noalign{\smallskip}
$N$ & $h_x$ & $ h_y$ & $h_t$ & $e_{L^2}(N)$ & $ p_{L_2}(N)$ & $e_{L^\infty}(N)$ & $ p_{L_\infty}(N)$ \\
\noalign{\smallskip}
\hline
\noalign{\smallskip}
4  & 1/4  & 1/16  & 1/32  &  3.4940e-3 &  ---   & 3.7327e-3 & ---  \\
8  & 1/8  & 1/32  & 1/64  &  2.0533e-4 &  4.089 & 2.4078e-4 & 3.956 \\
16 & 1/16 & 1/64  & 1/128 &  1.2386e-5 &  4.051 & 1.5246e-5 & 3.981 \\
32 & 1/32 & 1/128 & 1/256 &  7.5548e-7 &  4.035 & 9.5043e-7 & 4.004 \\
\hline
\end{tabular}
\end{center}
\end{table}

\smallskip\par{\bf Example 3.} Finally, the wave propagation is studied in
the three-layer medium with the sound speeds $s_1$, $s_2$ and $s_3=s_1$ (unless otherwise stated) respectively in its left, middle and right layers of the same thickness.
Here we take $X=Y=3000$ $m$ $=3$ $km$.
The source is defined as the Ricker-type wavelet known in geophysics and given by
\[
 \varphi(x,y,t) = \delta (x-x_0, y-y_0) \sin (50 t) e^{-200 t^2},
\]
where $\delta (x-x_0, y-y_0)$ is the Dirac distribution located at the center of domain $(x_0, y_0) = (1500\, m, 1500\, m)$.
Also we take $u_0=u_1=0$.
It was shown in \cite{HLZ19} that the wave dynamics is complicated.
The computational challenges arise due to discontinuous coefficient $c^2$ and
the very non-smooth distributional source function $\varphi$.

\par We take $h_x= h_y =h=\frac{X}{N}$ with even $N$ and approximate $\delta (x-x_0,y-y_0)$ as the mesh delta-function that equals $h^{-2}$ at the node $(x_0,y_0)$ and 0 at other nodes according to \eqref{avereq}.

\par Let first  $s_1=1500$ and $s_2=1000$ $m/s$ as in \cite{HLZ19}.
Figure~\ref{fig12}(a) shows 1D profiles of waves at $y=1.5$ $km$ for various times in the three-layer medium.
At $t=0.25$, the wave moves still inside the middle layer only.
At $t=0.75$, the wave fronts have already passed the interfaces of layers, have decreased their amplitude and move through the left and right layers towards the boundary;
simultaneously, the reflected waves of much smaller amplitude move back inside the middle layer.
At $t=1.05$, both reflected waves collide and acquire larger amplitude.
Then they continue their movement as shown at $t=1.15$.

\par For comparison, Figure~\ref{fig12}(b) shows 1D profiles of waves at $y=1.5$ $km$
in the homogeneous medium for $s_1=s_2=1000$ $m/s$.
Now only the refraction wave exists and moves towards the boundary with a constant velocity;
the graphs on the both figures are the same at $t=0.25$.
\begin{figure}[ht!]\centering
\subfigure[]{
\includegraphics[width=0.48\linewidth]{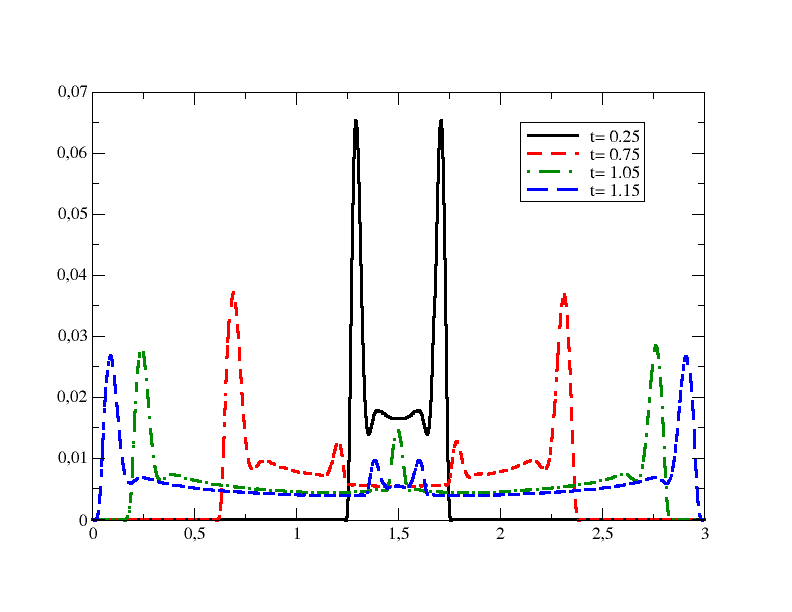}}
\subfigure[]{
\includegraphics[width=0.48\linewidth]{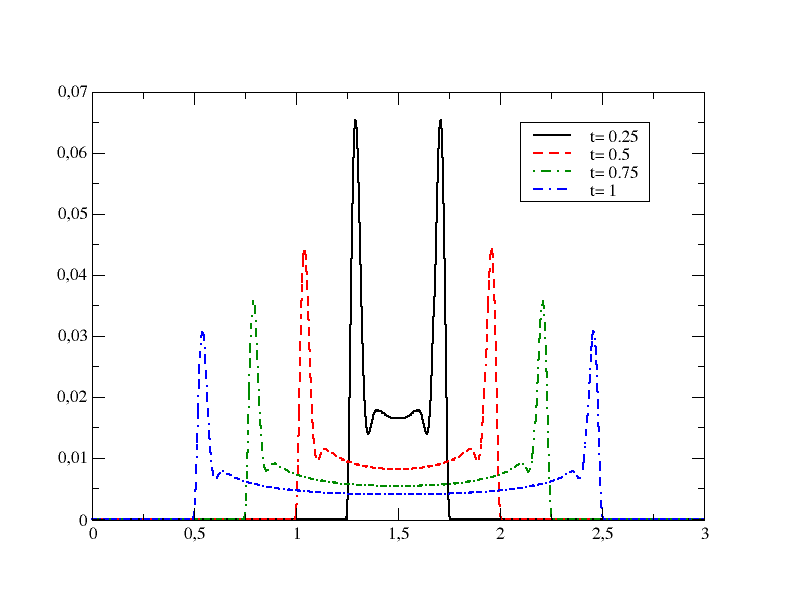}}
\caption{Dynamics of the waves at different times for: (a) the three-layer medium; (b) the homogeneous medium for $s_1=s_2=1000$ $m/s$}
\label{fig12}
\end{figure}

\par Next, in Figure~\ref{fig3} we present the dynamics of the waves
at $y=1.5$ $km$ in the case of three different sound speeds $s_1=1500$, $s_2=1000$ and $s_3=3000$.
At $t=0.25$, the graph is the same once again.
At $t=0.6$ and $t=0.7$, the wave fronts have already passed the interfaces of layers.
In contrast to Figure~\ref{fig12}, the amplitudes and speeds of the right refracted and reflected waves are higher than of the left ones.
\begin{figure}[ht]
\centerline{
\includegraphics[width=0.48
\textwidth]{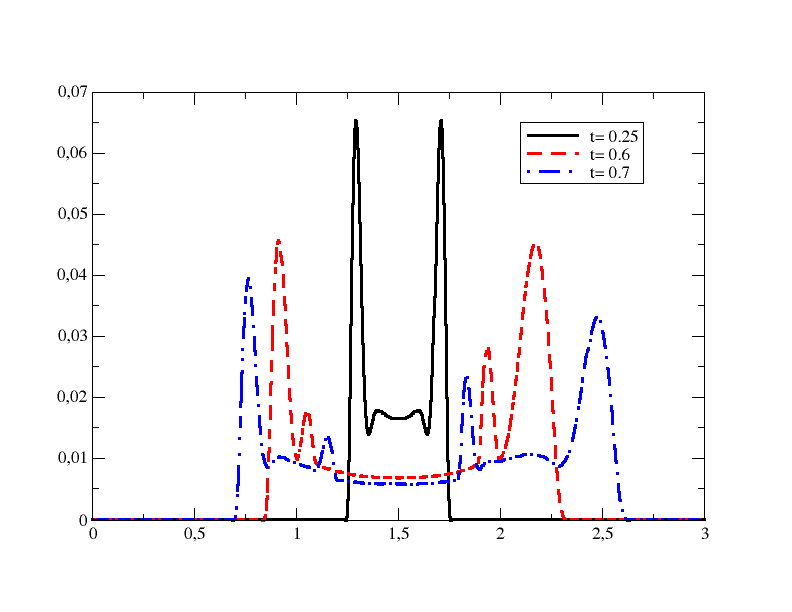}
}
\caption{Dynamics of the waves at different times
for the three-layer medium with $s_1=1500$, $s_2 = 1000$ and $s_3=3000$ $m/s$}
\label{fig3}
\end{figure}

\par In addition, we investigate experimentally the robustness of our
iterative method with respect to jumps in the sound speed and the convergence order of scheme $S_0$.
Such an analysis was not done in \cite{HLZ19}.
Table~\ref{tab2} contains the values of $N_{iter}$
for different speeds $s_1$ together with $s_2= 1000$ m/s.
In computations, the space steps are $h= 15$ and $7.5$ $m$; the time steps $h_t$ are respectively selected
from the stability requirement.
The presented results confirm that the iterative method \eqref{eqBAb2}
with the initial guess defined by \eqref{initial1} is both robust and fast.
\begin{table}[ht]
\begin{center}
\caption{Example 3:
$N_{iter}$ for different speeds $s_1$
in the left and right layers}
\label{tab2}
\vspace{2mm}
\begin{tabular}{ccclcclc}
\hline\noalign{\smallskip}
$s_1$ & $T$ & $h$ & $h_t$ & $N_{iter}$ & $h$ & $h_t$ & $N_{iter}$  \\
\noalign{\smallskip}
\hline
\noalign{\smallskip}
1000 & 1.0 & 15 & 0.005  & 9 & 7.5 & 0.0025 & 9 \\
1500 & 0.8 & 15 & 0.004  & 9 & 7.5 & 0.002  & 9 \\
3000 & 0.6 & 15 & 0.002  & 9 & 7.5 & 0.001  & 9 \\
6000 & 0.6 & 15 & 0.0012 & 9 & 7.5 & 0.0006 & 9 \\
\hline
\end{tabular}
\end{center}
\end{table}

\par Table~\ref{tab3} contains the errors $\bar{e}_{L^2}(N)$ and $e_{L^\infty}(N)$ in the mesh
scaled
$L_2$ and $L_\infty$ norms at $t= 0.8$, for $h =\frac{X}{N}$, with $N=100, 200, 400$, and $h_t=\frac{0.8}{N}$.
The approximations to these errors are computed as
\[
 \bar{e}_{L^2}(N) = \frac{1}{X}\|v_h - v_{h/2}\|_{L^2},\ \
 e_{L^\infty}(N) = \|v_h - v_{h/2}\|_{L^\infty},
\]
where $X$ equals the square root of the domain area, and $v_h$ is the solution to the scheme $S_0$ for $h =\frac{X}{N}$.
The computations are accomplished for the homogeneous case $s_1=s_2=1000$ $m/s$ and three-layer one with $s_1=1500$ and $s_2=1000$ $m/s$.
We see that since the exact solution is a non-smooth function, the convergence rates are essentially reduced, and they are visibly higher in a simpler case of the constant sound speed.
The results in $L^2$ norm are much better than in $L^\infty$ one.
Both of these last details are natural.
\begin{table}[ht]
\begin{center}
\caption{Example 3: errors $\bar{e}_{L^2}(N)$ and $e_{L^\infty}(N)$
and convergence rates $p_{L^q}(N)$  of
for a sequence of meshes and two speeds $s_1=1000$ and 1500 in the left and right layers}
\label{tab3}
\vspace{2mm}
\begin{tabular}{cccccccc}
\hline\noalign{\smallskip}
$s_1$ & $N$& $h$ & $h_t$ & $\bar{e}_{L^2}(N)$ &$p_{L_2}(N)$ & $e_{L^\infty}(N)$ & $ p_{L_\infty}(N)$  \\
\noalign{\smallskip}
\hline
\noalign{\smallskip}
1000 &100 & 30  & 0.008 &  1.78919e-3 & ---   & 0.012093 & ---   \\
1000 &200 & 15  & 0.004 &  4.04097e-4 & 2.146 & 0.004069 & 1.571  \\
1000 &400 & 7.5 & 0.002 &  9.88333e-5 & 2.032 & 0.001387 & 1.553 \\
\hline
1500 &100 & 30  & 0.008 &  2.01559e-3 & ---   & 0.012093 & --- \\
1500 &200 & 15  & 0.004 &  6.18800e-4 & 1.704 & 0.005448 & 1.150 \\
1500 &400 & 7.5 & 0.002 &  2.11363e-4 & 1.550 & 0.002736 & 0.994 \\
\hline
\end{tabular}
\end{center}
\end{table}

\par For comparison, we also investigate the accuracy of the standard explicit 2nd order scheme
$\Lambda_tz-c^2(\Lambda_x+\Lambda_y)z=\varphi$ for the same tests as given in Table~\ref{tab3}.
Table~\ref{tab5} contains the errors $\bar{e}_{L^2}(N)$
and $e_{L^\infty}(N)$ in the mesh scaled $L_2$ and $L_\infty$ norms at $t= 0.8$, for
$h=\frac{X}{N}$, $N=100, 200, 400$, and $h_t=\frac{0.8}{N}$.
Here the errors are computed as
\[
 \bar{e}_{L^2}(N) = \frac{1}{X}\|z_h- v_{h_0}\|_{L^2},\ \ \bar{e}_{L^\infty}(N)=\|z_h-v_{h_0}\|_{L^\infty},
\]
where $v_{h_0}$ is the solution of scheme $S_0$ for $h_0=\frac{X}{800}$ and $h_t = 0.001$ and
$z_h$ is the solution of the explicit 2nd order scheme.
Clearly, for the 2nd order scheme, the errors are larger and the convergence rates are worse than for scheme $S_0$, thus the latter scheme is better in the non-smooth case as well (the same practical conclusion for $n=1$ is done in \cite{ZK20}).
\begin{table}[ht]
\begin{center}
\caption{Example 3: errors
$\bar{e}_{L^2}(N)$ and $\bar{e}_{L^\infty}(N)$ and convergence rates $p_{L^q}(N)$ for
the standard explicit 2nd order scheme for a sequence of meshes and $s_1=1000$}
\label{tab5}
\vspace{2mm}
\begin{tabular}{cccccccc}
\hline\noalign{\smallskip}
$s_1$ & $N$& $h$ & $h_t$ & $\bar{e}_{L^2}(N)$ &$p_{L_2}(N)$ & $\bar{e}_{L^\infty}(N)$ & $ p_{L_\infty}(N)$  \\
\noalign{\smallskip}
\hline
\noalign{\smallskip}
1000 &200 & 15   & 0.004 & 2.57470e-3 & ---   & 0.015435 & ---  \\
1000 &400 & 7.5  & 0.002 & 9.75537e-4 & 1.400 & 0.008072 & 0.935 \\
1000 &800 & 3.75 & 0.001 & 3.18427e-4 & 1.615 & 0.004047 & 0.996 \\
\hline
\end{tabular}
\end{center}
\end{table}

\medskip\par\noindent
\textbf{Acknowledgements}
\medskip\par
The work of the first author was supported by the Russian Science Foundation, project no. 19-11-00169.

\medskip\noindent\textbf{Availability of Data and Materials} The datasets generated during the current study are available from the corresponding author on reasonable request. They support our published claims and comply with field standards.

\smallskip\noindent\textbf{Compliance with Ethical Standards}

\smallskip\noindent\textbf{Conflict of interest} There is no any conflict of interests/competing interests to declare that are relevant to the content of this article.

\smallskip\noindent\textbf{Code Availability} (software application or custom code) Our custom codes are not publicly available. They support our published claims and comply with field standards.

\end{document}